\documentclass{article}
\usepackage{graphicx}
\usepackage{amsmath, amssymb, amsthm}
\usepackage{xcolor}
\usepackage{dsfont}
\usepackage[hidelinks]{hyperref}

\usepackage[a4paper, left=2cm, right=2cm,
            top=2.5cm, bottom=2.5cm]{geometry}

\usepackage{tikz}
\usetikzlibrary{arrows.meta,decorations.markings}
\usetikzlibrary{arrows.meta,positioning}

\tikzset{
  nc point/.style={
    circle,fill=black,inner sep=1.5pt
  },
  kr point/.style={
    circle,draw=black,fill=white,inner sep=1.5pt
  },
  graph vertex/.style={
    circle,draw,fill=white,
    minimum size=6mm,inner sep=1pt,font=\small
  },
  graph edge/.style={
    semithick,
    postaction={decorate},
    decoration={
      markings,
      mark=at position 0.55 with {
        \arrow{Stealth[length=1.7mm,width=1.2mm]}
      }
    }
  }
}

\usepackage{sansmath}
\usepackage{caption}

\newcommand{\paperfigurefont}{%
  \sffamily\sansmath
  \captionsetup{
    font={small,sf},
    labelfont={bf,sf}
  }%
}

\newtheorem{theorem}             {Theorem} 
\newtheorem{lemma}      [theorem]{Lemma}
\newtheorem{proposition}[theorem]{Proposition}
\newtheorem{corollary}  [theorem]{Corollary}

\theoremstyle{definition}
\newtheorem{notation}   [theorem]{Notation} 
\newtheorem{remark}     [theorem]{Remark}
\newtheorem{definition} [theorem]{Definition}

\newtheorem{example}    [theorem]{Example}

\usepackage{authblk}

\title{Asymptotic infinitesimal freeness of  covariance matrices}
\author[1]{Daniel Munoz George}
\author[2]{Pei-Lun Tseng}

\affil[1]{Tecnol\'ogico de Monterrey, School of Engineering and Sciences,
Guadalajara, Mexico}

\affil[2]{Department of Applied Mathematics,
Chung Yuan Christian University,
Taoyuan, Taiwan}

\affil[ ]{\texttt{dmunozgeorge@gmail.com}}
\affil[ ]{\texttt{pltseng@cycu.edu.tw}}

\newcommand{\E}{\mathbb{E}}
\newcommand{\cP}{\mathcal{P}}
\newcommand{\NC}{\mathcal{NC}}
\newcommand{\Tr}{\text{Tr}}
\newcommand{\I}{\mathds{1}}

\begin{document}

\maketitle

\begin{abstract}
We consider $n\times n$ covariance matrices $M=\frac{1}{n}XX^*$ where $X=(x_{i,j})$ is a matrix whose entries are independent complex random variables with $\mathbb{E}(x_{i,j})=0$ and $\mathbb{E}(|x_{i,j}|^2)=1$. We derive a $\frac{1}{n}$ expansion of the mixed moments, $\frac{1}{n}\mathbb{E}(\Tr(M^{(r_1)}\cdots M^{(r_q)}))$, of the form $a_0+a_1\frac{1}{n}+O(\frac{1}{n^2})$. This permits us to find explicit formulas for the moments and infinitesimal moments of several covariance matrices where we allow repetition. As an application of our formulas, we derive asymptotic freeness and infinitesimal freeness of independent covariance matrices under a fourth-moment condition. This generalizes previous results for the Wishart ensemble in which $x_{i,j}$ is complex Gaussian.
\end{abstract}

\renewcommand{\contentsname}{Outline}

\tableofcontents
\bigskip

\section{Introduction and preliminaries}

The study of random matrices dates back to the work of Wishart and the seminal contributions of Wigner \cite{W1,W2}, who studied the eigenvalue distribution of Wigner random matrices, i.e., Hermitian random matrices whose entries on and above the diagonal
are independent and centered, with identically distributed entries
strictly above the diagonal. In his work, Wigner showed that, under mild moment conditions, certain universality results hold. More concretely,
\begin{equation}\label{Introduction Eq1}
\lim_{n\rightarrow\infty} \frac{1}{n}\mathbb{E}\left(\Tr(X_n^{2q})\right)
=\frac{1}{q+1}\binom{2q}{q},
\end{equation}
where $X_n$ is an $n\times n$ random matrix. The quantities on the right-hand side of Equation \ref{Introduction Eq1} are also known to be the moments of the semicircle distribution. Thus, this result establishes that the limiting spectral distribution of a Wigner matrix is the semicircle distribution. This work was later extended to the sample covariance matrix model \cite{MP}, where Marchenko and Pastur showed that the limiting spectral distribution is the Marchenko-Pastur distribution instead of the semicircle distribution. Since then, the limiting moments of more elaborate random matrix models have been extensively studied \cite{EKYY,TVW,JL,DE,BX,GWZ,BP}. 

A major development came in the 1980s with the seminal work of Voiculescu \cite{V}, in which he introduced free probability theory. In particular, Voiculescu introduced the concept of freeness; a rule for computing the mixed moments of several variables in terms of their individual moments. The theory gained considerable attention with the discovery that independent random matrices from several important ensembles become asymptotically free as their dimensions tend to infinity. Voiculescu was the first to establish that large-dimensional GUE matrices behave as freely independent random variables \cite{V}. This result was later extended to other models, including Wigner matrices \cite{D} and Wishart matrices \cite{HP,CC}. This connection has led to a fruitful interplay between random matrix theory and free probability, allowing limiting spectral distributions and asymptotic moments to be studied using algebraic and combinatorial methods.

Beyond the limiting normalized trace moments, it is natural to
investigate their first finite-size corrections. If
\begin{equation}\label{Equation: Moments of random matrix}
m_q := \lim_{n\to\infty}
\frac{1}{n}\mathbb{E}\bigl(\Tr(X_n^q)\bigr)
\end{equation}
exists, the corresponding infinitesimal moment is defined by
\begin{equation}\label{Introduction Equ2}
m_q^\prime :=
\lim_{n\to\infty}
n\left[
\frac{1}{n}\mathbb{E}\bigl(\Tr(X_n^q)\bigr)-m_q
\right],
\end{equation}
provided that this limit exists.
Corrections to the limiting moments of classical random matrix
ensembles have been investigated in several settings, including
Hermite and Laguerre ensembles and real Wishart matrices
\cite{J,DE2,Mingo19,MV}.
Their combinatorial structure is closely related to non-crossing
partitions and graph expansions.
In particular, Rahman, Munoz George and Mingo \cite{RMM25}
relate the $1/n$ corrections for the GOE and LOE to ribbon graphs
on the real projective plane and non-crossing annular pairings,
and also investigate the corresponding next-order corrections
for real and complex Gaussian ensembles.

A complementary direction concerns fluctuations, which are
described by joint cumulants of traces
\cite{MN,MS07}.
Munoz George and Mingo \cite{MGM22} computed the third-order
moments of complex Wigner matrices using quotient graphs and
partitioned permutations.
Mingo and Munoz George \cite{MMG24} subsequently studied
limiting trace cumulants of arbitrary fixed order and their
relations with higher-order free cumulants.
For products, Arizmendi, Munoz George and Sigarreta
\cite{AMGS25} obtained formulas for third-order free cumulants,
with applications to Gaussian Wishart matrices and products
of Ginibre matrices.
Bao and Munoz George \cite{BMG25} established bounds for
trace cumulants whose order may grow with the matrix dimension,
leading to quantitative central limit theorems for traces of
polynomials in Wigner and deterministic matrices.
These fluctuation results concern joint statistics of several
traces. The infinitesimal moments in
Equation~\ref{Introduction Equ2} instead describe the first
correction to a single expected normalized trace.

Infinitesimal free probability provides an algebraic framework
for studying limiting moments together with these first-order
corrections. Its analytic and combinatorial aspects were
developed by Belinschi and Shlyakhtenko \cite{BS} and
F\'evrier and Nica \cite{FN10}, respectively.
Infinitesimal freeness determines mixed moments and
infinitesimal mixed moments from the corresponding marginal
data.
In the operator-valued setting, Tseng \cite{T23} developed
a cumulant approach to infinitesimal freeness with amalgamation
and related it to ordinary operator-valued freeness over an
algebra of upper triangular matrices.
This approach also yields a characterization of infinitesimal
freeness through the vanishing of mixed cumulants.
Tseng \cite{T25} further studied infinitesimal multiplicative
convolutions and their transforms, including an application
to products of complex Wishart matrices.
These results provide a natural context for studying mixed
products of covariance matrices at the infinitesimal level.

Several random matrix models have been shown to exhibit
asymptotic infinitesimal freeness.
Mingo \cite{Mingo19} established this property for independent
complex Wishart matrices.
Further examples involve deterministic matrix families
conjugated by independent Haar unitary matrices, under
appropriate boundedness and convergence assumptions;
see Curran and Speicher \cite{CS11} and
C\'ebron, Dahlqvist and Gabriel \cite{CDG22}.
In the setting of finite-rank perturbations, Shlyakhtenko
\cite{Shl18} established asymptotic infinitesimal freeness
between certain Gaussian or unitarily invariant ensembles
and fixed finite-rank deterministic matrices.
Au \cite{Au21} obtained related results for Wigner matrices
and periodically banded GUE matrices under suitable
assumptions.
Mingo and Tseng \cite{MT24} developed an infinitesimal-operator
framework that encompasses several finite-rank matrix models
and allows distributions of commutators and anticommutators
to be computed.
They also identified connections between infinitesimal
moments and Boolean cumulants.
Another class of examples was obtained by Popa,
Szpojankowski and Tseng \cite{PST22}: matrices obtained
from a single GUE matrix by independent uniform entry
permutations are asymptotically infinitesimally free from
one another and from the original matrix, almost surely
with respect to the permutations.

In this paper, we study limiting mixed moments and
infinitesimal mixed moments of complex covariance matrices
of the form $\frac{1}{n}XX^*$.
Our analysis extends the complex Gaussian Wishart setting
to non-Gaussian entries.
Under the centering and variance normalization
$\mathbb{E}(x_{i,j})=0$ and
$\mathbb{E}(|x_{i,j}|^2)=1$, together with the independence
and moment-finiteness assumptions required for the
expansions, we obtain explicit formulas for the leading
term and the $1/n$ correction of expected normalized
traces of mixed products.
The leading term yields asymptotic freeness of independent
covariance matrices.
At the infinitesimal level, our formulas identify the
contributions of the entry moments and of the first-order
correction to the aspect ratio.
In particular, when
\[
\mathbb{E}(x_{i,j}^2)=0
\qquad\text{and}\qquad
\mathbb{E}|x_{i,j}|^4=2,
\]
the covariance matrices are asymptotically infinitesimally
free.
Thus, the infinitesimal freeness of independent complex
Wishart matrices persists for non-Gaussian entries satisfying
these second- and fourth-moment matching conditions.

\begin{definition}[Covariance matrix]
Let $n,p\in\mathbb{N}$. Let $X=(x_{i,j})_{\substack{i=1,\dots,n \\ j=1,\dots,p}}$ be an $n\times p$ random matrix whose entries are independent complex random variables with moments of all orders, such that $\E(x_{i,j})=0$ and $\E(|x_{i,j}|^2)=1$ for any $i,j$. The $n\times n$ random matrix
$$M=\frac{1}{n}XX^{*},$$
is called a \textit{Covariance matrix}. Here $*$ denotes the Hermitian conjugate of the matrix. We also let $p=p(n)$ be a function of $n$ such that $c=\lim_{n\to\infty}\frac{p}{n}$, and we call $c$ the \textit{parameter} of $M$.
\end{definition}

We will consider many covariance matrices $M^{(r_1)},\dots,M^{(r_q)}$ where we allow repetition. Let $\ker(r)\in\cP(q)$ denote the partition in which $u$ and $v$
belong to the same block if and only if $r_u=r_v$. If $ker(r)=\{\{1,\dots,q\}\}$ then all the matrices are the same while if $ker(r)=\{\{1\},\dots,\{q\}\}$ then all matrices are distinct. These cases represent the two extreme possibilities. 

To distinguish the entries of different matrices, we let
$$M^{(r_u)}=\frac{1}{n}X^{(r_u)}X^{(r_u)*},$$
where $X^{(r_u)}=(x_{i,j}^{(r_u)})_{\substack{i=1,\dots,n \\ j=1,\dots,p}}$ for $1\leq u\leq q$. We assume that all matrices have the same aspect-ratio limit $c$.
We are interested in the large $n$-limit of the expectation of the trace of products of covariance matrices. More formally, we are interested in the quantities
\begin{equation}\label{Equation: N moments}
\frac{1}{n}\E\circ \Tr(M^{(r_1)}\cdots M^{(r_q)})
\end{equation}
and their large $n$-limit
\begin{equation}\label{Equation: Limit moments}
\mu(M^{(r_1)},\dots, M^{(r_q)}) := \lim_{n\rightarrow\infty}\frac{1}{n}\E\circ \Tr(M^{(r_1)}\cdots M^{(r_q)}).
\end{equation}
The first quantity of Equation \ref{Equation: N moments} is called the $n$-moment (mixed moment) of the covariance matrices while the quantity of Equation \ref{Equation: Limit moments} is called the limiting moments, or simply moments. When $r_1=\cdots =r_q$ are all the same we recover the limiting moments described in Equation \ref{Equation: Moments of random matrix}.

Whenever the moments exist, we will be further interested in the infinitesimal mixed moments, defined by the limit
\begin{equation}\label{Equation: Limit inf moments}
\mu^\prime(M^{(r_1)},\dots, M^{(r_q)}) := \lim_{n\rightarrow\infty} n\left[\frac{1}{n}\E\circ \Tr(M^{(r_1)}\cdots M^{(r_q)})-\mu(M^{(r_1)},\dots, M^{(r_q)})\right].
\end{equation}


When $r_1=\cdots=r_q$, these quantities reduce to the
infinitesimal moments of a single covariance matrix,
as in Equation~\ref{Introduction Equ2}.
For the asymptotic independence results, we assume that
the entry families associated with distinct matrix labels
are independent. We write this as $M^{(r_1)},\dots, M^{(r_q)}$ is a family of independent covariance matrices. We also assume that the first-order correction to the
aspect ratio exists, and write
\begin{equation}
c^\prime :=
\lim_{n\to\infty}
n\left(\frac{p}{n}-c\right).
\end{equation}
Together, the limiting mixed moments $\mu$ and infinitesimal
mixed moments $\mu^\prime$ determine a joint infinitesimal
distribution.
The associated cumulant framework allows asymptotic
infinitesimal freeness to be characterized by the vanishing
of mixed free and infinitesimal free cumulants; see
\cite{FN10,T23} and
Section~\ref{Section: Inf free prob}.

Under the additional conditions
$\mathbb{E}(x_{i,j}^2)=0$ and
$\mathbb{E}|x_{i,j}|^4=2$, our results show that this
joint infinitesimal distribution is determined by
$c$, $c^\prime$ and the matrix labels.
Without these matching conditions, the infinitesimal
moments may also depend on the entry distribution;
in particular, the fourth entry moment contributes
to the $1/n$ correction.

\section{Main results}

First, we provide an explicit $\frac{1}{n}$ expansion of the mixed moments. Our first result provides such an expansion up to the constant term.

\begin{theorem}\label{Lemma: Expansion of moments 4}
Let $M^{(r_1)},\dots, M^{(r_q)}$ be a collection of independent covariance matrices, then
\begin{equation*}
\frac{1}{n}\E\circ \textup{\Tr}(M^{(r_1)}\cdots M^{(r_q)})=\sum_{\substack{\sigma\in \NC(q) \\ \sigma\leq ker(r)}}\left(\frac{p}{n}\right)^{\#(\sigma)}+O(\frac{1}{n}).
\end{equation*}
\end{theorem}

As an application of Theorem \ref{Lemma: Expansion of moments 4} we get explicit expressions for the limit mixed moments, cumulants and asymptotic freeness of independent covariance matrices.

\begin{corollary}\label{Corollary: Moments and cumulants}
Let $M^{(r_1)},\dots, M^{(r_q)}$ be a collection of independent covariance matrices, then all of the following statements hold.
\begin{enumerate}
    \item $\mu(M^{(r_1)},\dots,M^{(r_q)})=\sum_{\substack{\sigma\in \NC(q) \\ \sigma\leq ker(r)}}c^{\#(\sigma)}.$
    \item For any $n\geq 1$, $\kappa_n(M^{(r_{i_1})},\dots,M^{(r_{i_n})})=c$ whenever $r_{i_1}=\cdots =r_{i_n}$ and $0$ otherwise.
    \item The covariance matrices $M^{(r_1)},\dots, M^{(r_q)}$ are asymptotically free.
\end{enumerate}
\end{corollary}

Next, we go one step further, we derive an expansion up to order $\frac{1}{n}$. Before stating our result, let us briefly explain the notation. We defer a more precise description to Notation \ref{Notation: Distance in double tree}. We let $T$ be the graph with vertex set $\{1,\dots,2q\}$ and oriented edges
\[
e_1=(2,1),e_2=(3,2),\dots,e_{2q-1}=(2q,2q-1),e_{2q}=(1,2q).
\]
For a non-crossing pairing $\rho\in\NC_2(2q)$, we let $T^\rho$ be the graph obtained by merging the vertices of $T$ that belong to the same block of $\rho$. We will verify that $T^\rho$ is a fat tree; that is, after ignoring multiple edges, the resulting graph is a tree. Moreover, every pair of adjacent vertices is joined by exactly two edges, $e_u$ and $e_v$, where $u$ and $v$ have opposite parity.

Suppose that $A$ and $B$ are two vertices of $T^\rho$ at distance two. Then there exists a path from $A$ to $B$ consisting of two edges, which we denote by $e_a$ and $e_b$, where $e_a$ is adjacent to $A$ and $e_b$ is adjacent to $B$. We may assume that both $a$ and $b$ are even.

For example, if $q=3$ and
\[
\rho=\{\{1,5\},\{3\},\{2,4\},\{6\}\},
\]
then $T^\rho$ is the fat tree depicted in Figure \ref{fig:fat-tree}. If we choose $A=\{3\}$ and $B=\{1,5\}$, then the distance between $A$ and $B$ is two, which we write as
\[
d_{T^\rho}(A,B)=2.
\]
In this case, $e_a=e_2$ and $e_b=e_4$.

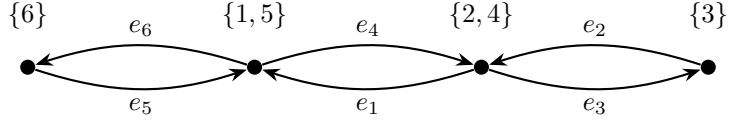
\begin{figure}
\centering
    \begin{tikzpicture}[
    >=Stealth,
    vertex/.style={circle,fill=black,inner sep=2pt}
]

\node[vertex] (v1) at (0,0) {};
\node[vertex] (v2) at (3,0) {};
\node[vertex] (v3) at (6,0) {};
\node[vertex] (v4) at (9,0) {};

\node[above=8pt of v1] {$\{6\}$};
\node[above=8pt of v2] {$\{1,5\}$};
\node[above=8pt of v3] {$\{2,4\}$};
\node[above=8pt of v4] {$\{3\}$};

\draw[->,thick]
    (v2) to[bend right=18]
    node[above] {$e_6$} (v1);

\draw[->,thick]
    (v1) to[bend right=18]
    node[below] {$e_5$} (v2);

\draw[->,thick]
    (v2) to[bend left=18]
    node[above] {$e_4$} (v3);

\draw[->,thick]
    (v3) to[bend left=18]
    node[below] {$e_1$} (v2);

\draw[->,thick]
    (v4) to[bend right=18]
    node[above] {$e_2$} (v3);

\draw[->,thick]
    (v3) to[bend right=18]
    node[below] {$e_3$} (v4);

\end{tikzpicture}
    \caption{A fat tree}
    \label{fig:fat-tree}
\end{figure}

\begin{theorem}\label{Lemma: Expansion of moments 5}
Let $M^{(r_1)},\dots, M^{(r_q)}$ be a collection of independent covariance matrices such that $\E(x_{1,2}^2)=0$, then
\begin{align*}
&\frac{1}{n}\E\circ \textup{\Tr}(M^{(r_1)}\cdots M^{(r_q)})\\
&=\sum_{\substack{\sigma\in \NC(q) \\ \sigma\leq ker(r)}}\left(\frac{p}{n}\right)^{\#(\sigma)}+ 
\frac{1}{n}\sum_{\substack{\sigma\in \NC(q) \\ \sigma\leq ker(r)}}\left(\frac{p}{n}\right)^{\#(\sigma)-1}\Bigg(\sum_{\substack{A,B\in 2\sigma \\ d_{T^\rho}(A,B)=2}}(\alpha_{a,b}-1)+\sum_{\substack{A,B\in 2Kr(\sigma)-1 \\ d_{T^\rho}(A,B)=2}}\frac{p}{n}(\alpha_{a,b}-1)\Bigg) +O(\frac{1}{n^2}).
\end{align*}
Here $\alpha_{a,b}=\frac{1}{2}\E|x_{1,2}|^4 \I_{r_{\frac{a}{2}}=r_{\frac{b}{2}}}+\I_{r_{\frac{a}{2}}\neq r_{\frac{b}{2}}}$, where $a,b$ are even numbers that depend on $A,B$ and are defined as in Notation \ref{Notation: Distance in double tree}. Further, $2\sigma$ is the partition whose blocks are obtained from multiplying each element of the blocks of $\sigma$ by 2 while $2Kr(\sigma)-1$ is the partition whose blocks are obtained from multiplying each element of the blocks of $Kr(\sigma)$ by 2 and subtracting 1. Finally $\rho=2\sigma\cup (2Kr(\sigma)-1)$.
\end{theorem}

\begin{remark}
In the present work, we also establish an expansion up to order $\frac{1}{n}$ without the additional assumption $\mathbb{E}(x_{1,2}^2)=0$. However, we prefer to defer the precise statement (see Lemma \ref{Lemma: Expansion of moments without pseudo variace zero}) to Section \ref{Section: Case of non zero psudovariance}, as it requires the introduction of additional notation.
\end{remark}

As a consequence of Theorem \ref{Lemma: Expansion of moments 5},
we obtain an explicit formula for the infinitesimal mixed moments
of independent covariance matrices satisfying
$\mathbb{E}(x_{1,2}^2)=0$.
If, in addition, $\mathbb{E}|x_{1,2}|^4=2$, these matrices
are asymptotically infinitesimally free.

\begin{corollary}\label{Corollary: infinitesimal moments}
Let $M^{(r_1)},\dots,M^{(r_q)}$ be a collection of independent covariance matrices such that $\E(x_{1,2}^2)=0$. Then all of the following statements hold.
\begin{enumerate}
    \item \begin{align*}
&\mu^\prime(M^{(r_1)},\dots,M^{(r_q)})\\
&=\sum_{\substack{\sigma\in \NC(q) \\ \sigma\leq ker(r)}}\#(\sigma)c^\prime c^{\#(\sigma)-1}+ 
\sum_{\substack{\sigma\in \NC(q) \\ \sigma\leq ker(r)}}c^{\#(\sigma)-1}\Bigg(\sum_{\substack{A,B\in 2\sigma \\ d_{T^\rho}(A,B)=2}}(\alpha_{a,b}-1)+\sum_{\substack{A,B\in 2Kr(\sigma)-1 \\ d_{T^\rho}(A,B)=2}}c(\alpha_{a,b}-1)\Bigg).    
    \end{align*}
    \item The covariance matrices are asymptotically infinitesimally free provided $\E|x_{1,2}|^4=2$.
\end{enumerate}
\end{corollary}


Finally, another consequence is the derivation of the free cumulants when all the covariance matrices are identical.

\begin{proposition}\label{Proposition: inf moments of covariant matrix}
Let \(M\) be the covariance matrix defined above, and assume that
\(\E[x_{1,2}^{2}]=0\). Then, for every \(q\geq 1\),
\begin{align*}
\mu^\prime(M^q)
=
\sum_{\sigma\in\NC(q)}
c^{\#(\sigma)-1}
\sum_{A\in\sigma}
\left(
c^\prime
+
c(\mathbb{E}|x_{1,2}|^4-2)\binom{|A|}{2}
\right).
\end{align*}
Equivalently, the corresponding free and infinitesimal free cumulants
are given by
\[
\kappa_n=c,
\qquad
\kappa_n^\prime
=
c^\prime
+
c(\mathbb{E}|x_{1,2}|^4-2)\binom{n}{2},
\qquad n\geq 1.
\]
\end{proposition}

\begin{remark}
In Proposition \ref{Proposition: inf moments of covariant matrix}, observe that if $M$ is Wishart, i.e., $x_{i,j}$ is complex Gaussian then the infinitesimal cumulants $\kappa_n^\prime=c^\prime$ for all $n\geq 1$, as obtained in previous works \cite[Theorem 30]{Mingo19}.
\end{remark}





Besides Sections 1 and 2 in which we introduce the preliminaries and state our results, in Sections \ref{Section: Non crossing partitions} and \ref{Section: Inf free prob} we give a brief recap on non-crossing partitions and infinitesimal free probability respectively. In Section \ref{Section: Graph theory} we introduce the graph machinery used throughout the paper while in Section \ref{Section: Asymptotic expansion of the moments} we derive an asymptotic expression for the mixed moments in terms of graphs. We apply this framework to prove
Theorem~\ref{Lemma: Expansion of moments 4}
in Section~\ref{Section: Moments}.
Finally, Section~\ref{Section: Infinitesimal moments}
contains the proofs of
Theorem~\ref{Lemma: Expansion of moments 5}
and Proposition~\ref{Proposition: inf moments of covariant matrix}.

\section{Non crossing partitions}\label{Section: Non crossing partitions}

Given \(n\in\mathbb N\), we denote
$
[n]:=\{1,2,\ldots,n\}.
$
A \emph{partition} \(\pi\) of \([n]\) is a collection of nonempty subsets
\[
\pi=\{V_1,V_2,\ldots,V_r\},
\]
called \emph{blocks}, such that
$
V_i\cap V_j=\varnothing,$ for all $i\neq j$, and 
$
V_1\cup V_2\cup\cdots\cup V_r=[n].
$
The set of all partitions of \([n]\) is denoted by \(\cP(n)\).
We order partitions by refinement: for partitions $\pi$ and
$\sigma$ of $[n]$, we write $\pi\leq\sigma$ if every block
of $\pi$ is contained in a block of $\sigma$.
In this case, we say that $\pi$ refines $\sigma$.

A partition \(\pi\in \cP(n)\) is called \emph{non-crossing} if there do not exist
\(1\le a<b<c<d\le n\) and two distinct blocks \(V,W\in\pi\) such that
\[
a,c\in V,
\qquad
b,d\in W.
\]
The collection of all non-crossing partitions of \([n]\) is denoted by
$\NC(n).$

For example,
\[
\{\{1,4\},\{2,3\}\}\in \NC(4),
\]
whereas
\[
\{\{1,3\},\{2,4\}\}\notin \NC(4).
\]

We denote by \(\cP_2(n)\) the set of all pair partitions of \([n]\), that is, partitions whose blocks all have cardinality two. Furthermore, we define
\[
\NC_2(n):=\cP_2(n)\cap \NC(n),
\]
the set of all non-crossing pair partitions of \([n]\).
In particular,
\[
\cP_2(n)=\NC_2(n)=\varnothing
\]
whenever \(n\) is odd.



Throughout, we identify each partition \(\pi\in\cP(n)\) with the
permutation whose cycles are its blocks, with the elements of each
block listed in increasing order. We write \(\#(\pi)\) for the number
of blocks of \(\pi\), or equivalently, the number of cycles of the
associated permutation.

Let
\[
\gamma=(1\,2\,\cdots\,n)
\]
denote the full cycle on \([n]\). We shall repeatedly use the following
characterization of noncrossing partitions, due to Biane
\cite{Biane97}:
\begin{equation}\label{eq:Biane-characterization}
\#(\pi)+\#(\pi^{-1}\gamma)
\leq n+1,
\qquad \pi\in\cP(n).
\end{equation}
Moreover, equality holds in \eqref{eq:Biane-characterization} if and
only if \(\pi\in\NC(n)\).

For \(\pi\in\NC(n)\), we define its left Kreweras complement by
\[
\operatorname{Kr}(\pi):=\gamma\pi^{-1}.
\]
Under the same convention, the right Kreweras complement is given by
\[
\operatorname{Kr}_{\mathrm R}(\pi):=\pi^{-1}\gamma.
\]
The two complements are related by
\[
\operatorname{Kr}(\pi)
=
\gamma\operatorname{Kr}_{\mathrm R}(\pi)\gamma^{-1}.
\]
Since conjugation by \(\gamma\) amounts to a cyclic relabeling of
\([n]\), it preserves both noncrossingness and the number of cycles.
Consequently,
\[
\operatorname{Kr}(\pi)\in\NC(n)
\qquad\text{and}\qquad
\#\bigl(\operatorname{Kr}(\pi)\bigr)+\#(\pi)=n+1.
\]

In terms of the topology of the non-crossing partitions, the right Kreweras complement map has a deeper interpretation, see \cite[Exercise 18.25]{NS06}. Given $\pi\in \NC(n)$, consider additional numbers $\{1',\dots,n'\}$. Then, the \emph{right Kreweras complement of  $\pi$}, is the partition $\operatorname{Kr}_{\mathrm R}(\pi)\in \NC(\{1',\dots,n'\})\cong \NC(n)$ that is the largest element among those $\sigma\in\NC(\{1',\dots,n'\})$ which have the property that $\pi\cup \sigma\in \NC(\{1,1',\dots,n,n'\})$. Analogously, the \emph{left Kreweras complement of  $\pi$}, is the partition $\operatorname{Kr}(\pi)\in \{1',\dots,n'\}\cong \NC(n)$ that is the largest element among those $\sigma\in\NC(\{1',\dots,n'\})$ which have the property that $\pi\cup \sigma\in \NC(\{1',1,\dots,n',n\})$. See, for instance, a graphical representation of this in Figure \ref{fig:kreweras-left-right}.

\begin{figure}[!h]
\centering
\paperfigurefont
\begin{minipage}{0.48\linewidth}
\centering
\begin{tikzpicture}[font=\small]
  \draw[gray!35] (0,0) circle (1.2);

  \foreach \k in {1,...,4} {
    \coordinate (P\k) at ({90-90*(\k-1)}:1.2);
    \coordinate (Q\k) at ({45-90*(\k-1)}:1.2);
  }

  \draw[thick] (P1)--(P3);
  \draw[thick,dashed] (Q1)--(Q2) (Q3)--(Q4);

  \foreach \k in {1,...,4} {
    \node[nc point] at (P\k) {};
    \node[kr point] at (Q\k) {};
    \node at ({90-90*(\k-1)}:1.48) {$\k$};
    \node at ({45-90*(\k-1)}:1.48) {$\k'$};
  }

  \node at (0,-1.95) {Right complement};
\end{tikzpicture}
\end{minipage}\hfill
\begin{minipage}{0.48\linewidth}
\centering
\begin{tikzpicture}[font=\small]
  \draw[gray!35] (0,0) circle (1.2);

  \foreach \k in {1,...,4} {
    \coordinate (P\k) at ({45-90*(\k-1)}:1.2);
    \coordinate (Q\k) at ({90-90*(\k-1)}:1.2);
  }

  \draw[thick] (P1)--(P3);
  \draw[thick,dashed] (Q1)--(Q4) (Q2)--(Q3);

  \foreach \k in {1,...,4} {
    \node[nc point] at (P\k) {};
    \node[kr point] at (Q\k) {};
    \node at ({45-90*(\k-1)}:1.48) {$\k$};
    \node at ({90-90*(\k-1)}:1.48) {$\k'$};
  }

  \node at (0,-1.95) {Left complement};
\end{tikzpicture}
\end{minipage}

\caption{The right and left Kreweras complements of
$\pi=\{\{1,3\},\{2\},\{4\}\}$.
Solid chords and filled vertices represent $\pi$;
dashed chords and open vertices represent its complement.
On the primed labels,
$\operatorname{Kr}_{\mathrm R}(\pi)
=\{\{1',2'\},\{3',4'\}\}$ and
$\operatorname{Kr}(\pi)
=\{\{1',4'\},\{2',3'\}\}$.
The circular orders are read clockwise.}
\label{fig:kreweras-left-right}
\end{figure}
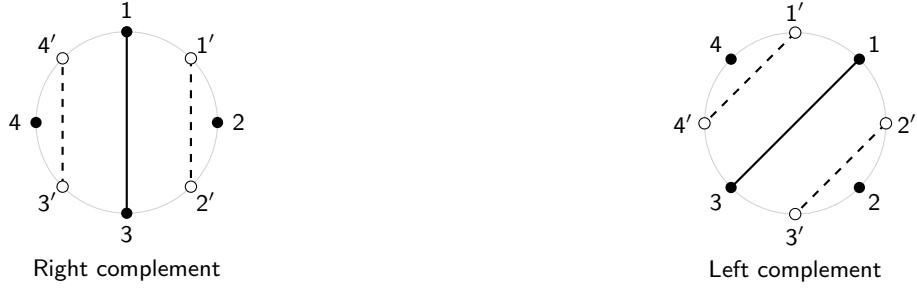

Historically, the right Kreweras complement has appeared in many computations involving the moments of various random matrix models. In this case, we show that the left Kreweras complement arises naturally. Thus, from now on, whenever we refer to the Kreweras complement, we mean the left Kreweras complement.

If $\pi\in \NC(\{2,4,\dots,2n\})$ is a non-crossing partition over the set of even numbers then we let $\operatorname{Kr}(\pi)$ be the Kreweras complement of $\pi$ over the set $\{1,3,\dots,2n-1\}$ which is the largest non-crossing partition such that $\pi\cup \operatorname{Kr}(\pi)\in \NC(2n)$. An example of this can be seen in Figure \ref{fig:kreweras-even-odd}.

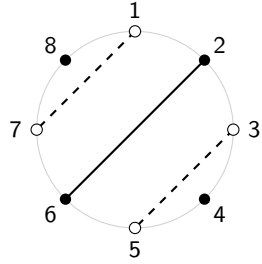
\begin{figure}[htbp]
\centering
\paperfigurefont
\begin{tikzpicture}[font=\small]
  \draw[gray!35] (0,0) circle (1.3);

  \foreach \k in {1,...,8}
    \coordinate (P\k) at ({90-45*(\k-1)}:1.3);

  \draw[thick] (P2)--(P6);

  \draw[thick,dashed] (P1)--(P7) (P3)--(P5);

  \foreach \k in {2,4,6,8}
    \node[nc point] at (P\k) {};

  \foreach \k in {1,3,5,7}
    \node[kr point] at (P\k) {};

  \foreach \k in {1,...,8}
    \node at ({90-45*(\k-1)}:1.58) {$\k$};
\end{tikzpicture}

\caption{The even partition
$\pi=\{\{2,6\},\{4\},\{8\}\}$ (solid)
and its left Kreweras complement
$\operatorname{Kr}(\pi)=\{\{1,7\},\{3,5\}\}$ (dashed).
Their union is a non-crossing partition of $[8]$.
This is the left-hand convention of
Figure~\ref{fig:kreweras-left-right}, with
$j\mapsto 2j$ and $j'\mapsto 2j-1$.}
\label{fig:kreweras-even-odd}
\end{figure}

\section{Infinitesimal free probability}\label{Section: Inf free prob}

In this section, we briefly recall the notions of freeness and infinitesimal
freeness which will be used throughout the paper. We also fix the notation for
free cumulants and infinitesimal free cumulants.

\subsection{The notion of freeness}

A \emph{non-commutative probability space} is a pair $(\mathcal A,\varphi)$,
where $\mathcal A$ is a unital algebra over $\mathbb C$ and
$\varphi:\mathcal A\to \mathbb C$ is a unital linear functional, that is,
$\varphi(1_{\mathcal A})=1$. Elements of $\mathcal A$ are called
non-commutative random variables.

Let $a_1,\ldots,a_k\in\mathcal A$. The \emph{joint distribution} of
$(a_1,\ldots,a_k)$ is the linear functional
\[
\mu_{a_1,\ldots,a_k}:
\mathbb{C}\langle X_1,\ldots,X_k\rangle
\to
\mathbb C
\]
defined by
\[
\mu_{a_1,\ldots,a_k}(p)
=
\varphi\bigl(p(a_1,\ldots,a_k)\bigr),
\qquad
p\in\mathbb C\langle X_1,\ldots,X_k\rangle.
\]

Let $(\mathcal A,\varphi)$ be a non-commutative probability space. A family of
unital subalgebras $(\mathcal A_i)_{i\in I}$ is said to be \emph{freely independent} (or \emph{free} for short) if
$$
\varphi(a_1a_2\cdots a_m)=0
$$
whenever $a_j\in \mathcal A_{i_j}$ satisfies
$
\varphi(a_j)=0$, $1\le j\le m,$ and the indices are alternating, namely
\[
i_1\neq i_2,\quad i_2\neq i_3,\quad \ldots,\quad i_{m-1}\neq i_m.
\]
A family of random variables $(x_i)_{i\in I}\subset \mathcal A$ is called free
if the unital subalgebras generated by them are free.

In \cite{Speicher94}, Speicher introduced the notion of free cumulants and showed that they provide a combinatorial characterization of freeness, which we briefly recall below. 

The \emph{free cumulants}
$
\{\kappa_n:\mathcal A^n\to \mathbb C\}_{n\geq 1}
$
are the multi-linear functionals defined by the moment-cumulant formula
\[
\varphi(a_1a_2\cdots a_n)
=
\sum_{\pi\in NC(n)}
\kappa_\pi(a_1,\ldots,a_n),
\]
where $NC(n)$ denotes the lattice of non-crossing partitions of $[n]$, and
\[
\kappa_\pi(a_1,\ldots,a_n)
=
\prod_{V\in \pi}
\kappa_{|V|}(a_{i_1},\ldots,a_{i_{|V|}})
\]
for each block $V=\{i_1<\cdots<i_{|V|}\}$ of $\pi$.

The fundamental characterization of freeness in terms of free cumulants is the
following: a family of subalgebras $(\mathcal A_i)_{i\in I}$ is free if and only
if all mixed free cumulants vanish. More precisely,
\[
\kappa_n(a_1,\ldots,a_n)=0
\]
whenever $a_j\in \mathcal A_{i_j}$ and the variables do not all come from the
same subalgebra. Equivalently, non-vanishing free cumulants can only involve
variables from one common free component.

\subsection{The notion of infinitesimal freeness}

We now recall the infinitesimal analogue of the preceding notions. An
\emph{infinitesimal non-commutative probability space} is a triple
$(\mathcal A,\varphi,\varphi'),$
where $(\mathcal A,\varphi)$ is a non-commutative probability space and
$\varphi':\mathcal A\to\mathbb C$ is a linear functional satisfying
$
\varphi'(1_{\mathcal A})=0.
$
For $a_1,\dots,a_k\in \mathcal{A}$, the pair of linear functionals $(\mu,\mu')$ which both map $\mathbb{C}\langle X_1,\dots,X_k\rangle$ into $\mathbb{C}$ is the \textit{joint infinitesimal distribution} of $\{a_1,\dots,a_k\}$ if $\mu=\mu_{a_1,\dots,a_k}$ and
$$
\mu'(p)=\varphi'(p(a_1,\dots,a_k)).
$$
Similarly, we denote the linear functional $\mu'$ by $\mu'_{a_1,\dots,a_k}.$

A family of unital subalgebras $(\mathcal A_i)_{i\in I}$ is said to be
\emph{infinitesimally free} if it is free with respect to $\varphi$ and, in
addition, for every $n\ge1$, every choice of
$i_1,\ldots,i_n\in I$ with
$i_1\neq i_2\neq\cdots\neq i_n$, and
$a_j\in\mathcal A_{i_j}$ satisfying
$\varphi(a_j)=0$, one has
\[
\varphi'(a_1\cdots a_n)
=
\sum_{j=1}^n
\varphi'(a_j)\,
\varphi(a_1\cdots a_{j-1}a_{j+1}\cdots a_n).
\]
Similarly, a family of random variables $(x_i)_{i\in I}\subset \mathcal A$ is called infinitesimally free
if the unital subalgebras generated by them are infinitesimally free.

Motivated by the role of free cumulants in the characterization of freeness, F\'evrier and Nica introduced infinitesimal free cumulants and proved that infinitesimal freeness is equivalent to the vanishing of mixed ordinary and infinitesimal free cumulants (see \cite{FN10}). We briefly recall them below.

The \emph{infinitesimal free cumulants}
$\{\kappa_n':\mathcal A^n\to\mathbb C\}_{n\ge1}$
are the sequence of multilinear functionals determined via the following relation
\[
\varphi'(a_1\cdots a_n)
=
\sum_{\pi\in NC(n)}
\partial\kappa_\pi(a_1,\ldots,a_n),
\]
where
\[
\partial\kappa_\pi(a_1,\ldots,a_n)
=
\sum_{V\in\pi}
\kappa_{|V|}'
(a_{i_1},\ldots,a_{i_{|V|}})
\prod_{\substack{W\in\pi\\ W\neq V}}
\kappa_{|W|}
(a_{j_1},\ldots,a_{j_{|W|}}),
\]
with
\[
V=\{i_1<\cdots<i_{|V|}\},
\qquad
W=\{j_1<\cdots<j_{|W|}\}.
\]

Then F{\'e}vrier and Nica showed that a family of unital subalgebras $(\mathcal A_i)_{i\in I}$ is {infinitesimally free if and only if 
\[
\kappa_n(a_1,\ldots,a_n)=0, \quad \text{and}\quad \kappa_n'(a_1,\ldots,a_n)=0,
\]
for mixed families. 

\subsection{Random matrices and asymptotic infinitesimal freeness}

One of the most remarkable achievements of free probability is its intimate connection with random matrix theory. In this subsection, we briefly recall the aspects of this connection that are relevant to the present work.

Let $\{A_N\}_{N\ge1}$ be a sequence of $N\times N$ random matrices. We say that
$A_N$ converges in distribution to a non-commutative random variable
$a\in(\mathcal A,\varphi)$ if
\[
\lim_{N\to\infty}
\frac1N\mathbb E\circ\Tr(A_N^m)
=
\varphi(a^m),
\qquad
m\ge1.
\]
More generally, if
$A_N^{(1)},\ldots,A_N^{(r)}$
are families of random matrices, we say that they converge jointly in
distribution to
$a_1,\ldots,a_r\in(\mathcal A,\varphi)$
if
\[
\lim_{N\to\infty}
\frac1N
\mathbb E\circ\Tr
\bigl(
P(A_N^{(1)},\ldots,A_N^{(r)})
\bigr)
=
\varphi
\bigl(
P(a_1,\ldots,a_r)
\bigr)
\]
for every non-commutative polynomial $P$.


To capture the first-order correction beyond the limiting distribution, one
considers the infinitesimal distribution. Let
$(\mathcal A,\varphi,\varphi')$
be an infinitesimal non-commutative probability space. We say that
$A_N$ converges infinitesimally to
$a\in(\mathcal A,\varphi,\varphi')$
if
\[
\lim_{N\to\infty}
\frac1N\mathbb E\circ\Tr(A_N^m)
=
\varphi(a^m),
\]
and
\[
\lim_{N\to\infty}
N\left(
\frac1N\mathbb E\circ\Tr(A_N^m)
-
\varphi(a^m)
\right)
=
\varphi'(a^m),
\qquad
m\ge1.
\]
Similarly, a family
$\{A_N^{(1)},\ldots,A_N^{(r)}\}$
is said to converge jointly in infinitesimal distribution if the above
convergence holds for every non-commutative polynomial in the variables.

We say that the families
\[
A_N^{(1)},\ldots,A_N^{(r)}
\]
are \emph{asymptotically free} (resp., \emph{asymptotically infinitesimally free})
if their joint (resp., joint infinitesimal) distribution converges to a family
\[
\{a_1,\ldots,a_r\}
\]
of free random variables in a non-commutative probability space
$(\mathcal A,\varphi)$
(resp., infinitesimally free random variables in an infinitesimal
non-commutative probability space
$(\mathcal A,\varphi,\varphi')$).

\section{Graph theory}\label{Section: Graph theory}

By an unoriented graph or simply a graph we mean a pair $(V,E)$ where $V$, the vertices, is a finite non-empty set, and $E$  is the finite set of edges. Each edge, $e$, will have \textit{endpoints} $v_1, v_2 \in V$, and we shall say $e$ \textit{connects} the vertices $v_1$ and $v_2$. We allow $v_1=v_2$, in which case we say that $e$ is a \textit{loop}. In addition, we allow two edges $e$ and $f$ to have the same endpoints without being equal; this is often described by saying the graph has \textit{multiple edges}.

We shall also consider \textit{oriented} graphs; by this we mean that each edge has an origin and a terminus. For an edge $e$, we denote these by $sr(e)$ and $tr(e)$ respectively. To distinguish  between two edges connecting the same pair of vertices we frequently label the edges with the numbers $1,\dots,|E|$. We shall frequently use the same numbers to label vertices and edges; thus, sometimes we may require the notation $e_u$ to denote the edge whose label is $u$. We do not distinguish between $e_u$ and $u$; so a partition of the edges can have blocks of the form $\{i_1,\dots,i_t\}$ or $\{e_{i_1},\dots,e_{i_t}\}$, whichever is more convenient.

\begin{definition}[Quotient graph]\label{Definition: Quotient graph}
Let $G=(V,E)$ be a graph either oriented  or unoriented and let $\pi\in \cP(V)$ be a partition of $V$. The \textit{quotient graph} $G^{\pi}=(V^{\pi},E^{\pi})$ is the labelled oriented or unoriented graph obtained after identifying vertices in the same block. For $u\in V$, $[u]_{\pi}$ denotes the block of $\pi$ that contains $u$, the vertices of $G^{\pi}$ consist of the blocks of $\pi$. The edges of $G$ become the edges of $G^\pi$. Namely, if $e$ is an edge of $G$ connecting $u$ to $v$, then there will be a corresponding edge in $E^\pi$ connecting $[u]_{\pi}$ to $[v]_{\pi}$. We use the notation $e^\pi$ only when it is necessary to emphasize that we are referring to an element of $E^\pi$. 

From a partition $\pi \in \cP(V)$ we always obtain another partition $\overline{\pi} \in \cP(E)$ by saying $e \sim_{\overline{\pi}} f$ whenever $e^\pi$ and $f^\pi$ connect the same pair of vertices in $G^\pi$.
\end{definition}

\begin{definition}[Elementarization graph]
Given an oriented or unoriented graph $G=(V,E)$ its \textit{elementarization} is the unoriented graph $\overline{G}=(V,\overline{E})$ obtained from forgetting multiple edges and orientation of the edges of $G$. The vertices of $\overline{G}$ and $G$ are the same while the vertices of $\overline{E}$ consist of equivalence classes, $\overline{e}$, where $\overline{e}=\{e_1,\dots,e_n\}$ consists of all edges $e_i\in E$ joining the same pair of vertices. The \textit{multiplicity} of an edge $\overline{e}$ is the size of the equivalence class. 
\end{definition}

Since we are working with traces, throughout this paper we focus on a very particular graph: a cycle graph with an even number of edges. Throughout the remainder of this section, we restrict our attention to this graph and develop several results specific to it.

\begin{notation}\label{Notation: Graph m cycle}
Let $m$ be an even number and let $T_m=(V,E)$ be the oriented graph with set of vertices $V=[m]$ and edges $e_u=(\gamma(u),u)$ for $1\leq u\leq m$. Here $\gamma=(1,\dots,m)\in S_m$.
\end{notation}

For the rest of the paper we take $m$ to be even and $T=T_m$. Our first goal is to characterize the partitions $\pi$ for which the graph $\overline{T^\pi}$ has no edges of multiplicity $1$ and $\#(\pi)=m/2+1$.

\begin{definition}\label{Definition: Double tree}
Let $\pi\in \cP(m)$. We say that the graph $T^\pi$ is a \textit{double tree} if $\overline{T^\pi}$ is a tree with any edge of multiplicity $2$. 
\end{definition}

\begin{lemma}\label{Lemma: First characterization of pi}
Let $\pi\in \cP(m)$ be such that $\overline{T^\pi}$ has no edges of multiplicity $1$. Then
$$\#(\pi)\leq m/2+1,$$
with equality if and only if $\overline{T^\pi}$ is a double tree.
\end{lemma}
\begin{proof}
On the one hand, $\overline{T^\pi}$ has $\#(\pi)$ vertices and $\#(\overline{\pi})$ edges, therefore as it is a connected graph,
\begin{equation}\label{Aux 1}
\#(\pi)-\#(\overline{\pi})\leq 1,
\end{equation}
with equality if and only if $\overline{T^\pi}$ is a tree. On the other hand, as any edge of $\overline{T^\pi}$ has multiplicity at least $2$ it follows
\begin{equation}\label{Aux 2}
\#(\overline{\pi})-m/2\leq 0,
\end{equation}
with equality if and only if any edge of $\overline{T^\pi}$ has multiplicity $2$. Combining inequalities \ref{Aux 1} and \ref{Aux 2} yields
$$\#(\pi)-m/2\leq 1,$$
with equality if and only if $\overline{T^\pi}$ is a tree with any edge of multiplicity $2$.
\end{proof}

\begin{remark}\label{Remark: Properties of double tree}
Observe that if $\overline{T^\pi}$ is a double tree, then for any pair of edges $e_u,e_v$ joining the same pair of vertices of $T^\pi$ it is satisfied that $e_u$ and $e_v$ have the opposite orientation. This follows immediately from the fact that each vertex of $\overline{T^\pi}$ has the same number of incoming and outgoing edges and $\overline{T^\pi}$ is a tree.
\end{remark}

\begin{lemma}\label{Lemma: Gamma sigma overline less or equal than pi}
Let $\pi\in \cP(m)$ and $\sigma\in \cP(m)$ be such that all the blocks of $\sigma$ have size $2$. If for any block $\{u,v\}\in \sigma$ the edges $e_u$ and $e_v$ connect the same pair of vertices of $T^{\pi}$ and with the opposite orientation, then
$$\gamma\sigma \leq \pi.$$
\end{lemma}
\begin{proof}
If $\{u,v\}$ is a block of $\sigma$ then $e_u$ goes from $[\gamma(u)]_{\pi}$ to $[u]_{\pi}$ in $T^{\pi}$. Similarly $e_v$ goes from $[\gamma(v)]_{\pi}$ to $[v]_{\pi}$. By hypothesis,
$$[u]_{\pi}=[\gamma(v)]_{\pi}.$$
Thus,
$$[\gamma\sigma(u)]_{\pi}=[\gamma(v)]_{\pi}=[u]_{\pi},$$
which proves that $u$ and $\gamma\sigma(u)$ are in the same block of $\pi$.
\end{proof}

\begin{lemma}\label{Lemma: Second characterization of pi}
Let $\pi\in \cP(m)$ be such that $\overline{T^\pi}$ has no edges of multiplicity $1$. Then
$$\#(\pi)\leq m/2+1,$$
with equality if and only if $\overline{\pi}\in \NC_2(m)$ and $\pi=\gamma\overline{\pi}$. Here $\gamma=(1,\dots,m)\in S_m$.
\end{lemma}
\begin{proof}
The inequality was already proved in Lemma \ref{Lemma: First characterization of pi}. Let us prove the ``if and only if". If $\#(\pi)=m/2+1$ by Lemma \ref{Lemma: First characterization of pi} it follows $\overline{\pi}\in \cP_2(m)$, further as pointed out in remark \ref{Remark: Properties of double tree}; any block $\{u,v\}\in\overline{\pi}$ is such that $e_u$ and $e_v$ connect the same pair of vertices of $\overline{T^\pi}$ with the opposite orientation. Then by Lemma \ref{Lemma: Gamma sigma overline less or equal than pi} it follows that
$$\gamma\overline{\pi}\leq \pi,$$
thus
$$m+1=\#(\pi)+\#(\overline{\pi})\leq \#(\gamma\overline{\pi})+\#(\overline{\pi})\leq m+1,$$
where the first equality follows by hypothesis and the last inequality follows from inequality \ref{eq:Biane-characterization}. We conclude $\#(\pi)=\#(\gamma\overline{\pi})=m+1-\#(\overline{\pi})$ and therefore $\overline{\pi}\in \NC_2(m)$ and $\pi=\gamma\overline{\pi}$. Conversely, suppose $\overline{\pi}\in \NC_2(m)$ and $\pi=\gamma\overline{\pi}$, then
$$\#(\pi)=\#(\gamma\overline{\pi})=m+1-\#(\overline{\pi})=m/2+1.$$
\end{proof}

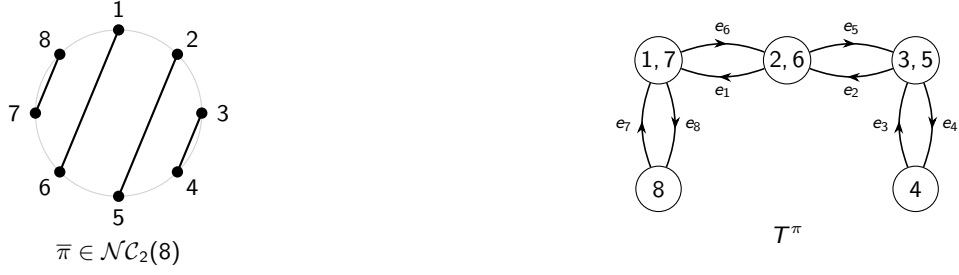
\begin{figure}[htbp]
\centering
\paperfigurefont
\begin{minipage}{0.43\linewidth}
\centering
\begin{tikzpicture}[font=\small]
  \draw[gray!35] (0,0) circle (1.1);

  \foreach \k in {1,...,8}
    \coordinate (P\k) at ({90-45*(\k-1)}:1.1);

  \foreach \u/\v in {1/6,2/5,3/4,7/8}
    \draw[thick] (P\u)--(P\v);

  \foreach \k in {1,...,8} {
    \node[nc point] at (P\k) {};
    \node at ({90-45*(\k-1)}:1.38) {$\k$};
  }

  \node at (0,-1.85)
    {$\overline{\pi}\in\mathcal{NC}_2(8)$};
\end{tikzpicture}
\end{minipage}\hfill
\begin{minipage}{0.53\linewidth}
\centering
\begin{tikzpicture}[font=\scriptsize]
  \node[graph vertex] (A) at (0,0) {$1,7$};
  \node[graph vertex] (B) at (1.7,0) {$2,6$};
  \node[graph vertex] (C) at (3.4,0) {$3,5$};
  \node[graph vertex] (D) at (3.4,-1.7) {$4$};
  \node[graph vertex] (E) at (0,-1.7) {$8$};

  \draw[graph edge] (B) to[bend left=20]
    node[midway,below] {$e_1$} (A);
  \draw[graph edge] (A) to[bend left=20]
    node[midway,above] {$e_6$} (B);

  \draw[graph edge] (C) to[bend left=20]
    node[midway,below] {$e_2$} (B);
  \draw[graph edge] (B) to[bend left=20]
    node[midway,above] {$e_5$} (C);

  \draw[graph edge] (D) to[bend left=20]
    node[midway,left] {$e_3$} (C);
  \draw[graph edge] (C) to[bend left=20]
    node[midway,right] {$e_4$} (D);

  \draw[graph edge] (E) to[bend left=20]
    node[midway,left] {$e_7$} (A);
  \draw[graph edge] (A) to[bend left=20]
    node[midway,right] {$e_8$} (E);

  \node[font=\small] at (1.7,-2.3) {$T^\pi$};
\end{tikzpicture}
\end{minipage}

\caption{The pairing
$\overline{\pi}
=\{\{1,6\},\{2,5\},\{3,4\},\{7,8\}\}$
and the corresponding double tree.
For $\gamma=(1\,2\,\cdots\,8)$,
$\pi=\gamma\overline{\pi}
=\{\{1,7\},\{2,6\},\{3,5\},\{4\},\{8\}\}$.
Entries inside each vertex specify its block of $\pi$.
The orientations follow
$e_u=([\gamma(u)]_\pi,[u]_\pi)$.
In particular, $\#(\pi)=5=8/2+1$.}
\label{fig:double-tree-pairing}
\end{figure}

\begin{remark}\label{Remark: Relation double tree and non-crossing pairings}
From Lemmas \ref{Lemma: First characterization of pi} and \ref{Lemma: Second characterization of pi} it is clear that there is a bijection between the set of double trees and the set of non-crossing pairings. Such a bijection is given by the relation $\pi=\gamma\overline{\pi}$ where $\overline{\pi}\in \NC_2(m)$. An explicit example of this can be seen in Figure \ref{fig:double-tree-pairing}.
\end{remark}

So far, we have seen several results that characterize the case when $\#(\pi)=m/2+1$, now we turn our attention to the case when $\#(\pi)=m/2$. At this point, recall that a connected graph with a unique cycle is a graph with the same number of vertices and edges. Equivalently, the graph has a unique cycle and removing any edge of this cycle results in a tree.

\begin{definition}\label{Definition: Double uniciclyc and 2-4 tree type}
Let $\pi\in \cP(m)$.
\begin{enumerate}
    \item We say that the graph $T^\pi$ is a \textit{2-4 tree} if $\overline{T^\pi}$ is a tree with any edge of multiplicity $2$ except for one edge that has multiplicity $4$.
    \item We say that the graph $T^\pi$ is a \textit{double unicycle} if $\overline{T^\pi}$ is a graph with a unique cycle with any edge of multiplicity $2$.
\end{enumerate}
Examples of 2-4 tree and double unicycle are given in Figure \ref{fig:two-four-tree-unicycle}.
\end{definition}

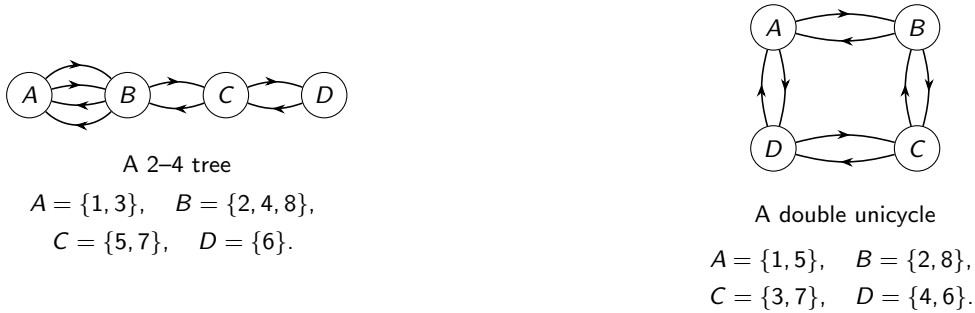
\begin{figure}[htbp]
\centering
\paperfigurefont
\begin{minipage}{0.48\linewidth}
\centering
\begin{tikzpicture}
  \node[graph vertex] (A) at (0,0) {$A$};
  \node[graph vertex] (B) at (1.3,0) {$B$};
  \node[graph vertex] (C) at (2.6,0) {$C$};
  \node[graph vertex] (D) at (3.9,0) {$D$};

  \foreach \angle in {15,45} {
    \draw[graph edge] (A)
      to[bend left=\angle] (B);
    \draw[graph edge] (B)
      to[bend left=\angle] (A);
  }

  \foreach \u/\v in {B/C,C/D} {
    \draw[graph edge] (\u)
      to[bend left=20] (\v);
    \draw[graph edge] (\v)
      to[bend left=20] (\u);
  }

  \node[font=\small] at (1.95,-0.9)
    {A $2$--$4$ tree};
\end{tikzpicture}

\smallskip
{\small
$\begin{gathered}
 A=\{1,3\},\quad B=\{2,4,8\},\\
 C=\{5,7\},\quad D=\{6\}.
\end{gathered}$
}
\end{minipage}\hfill
\begin{minipage}{0.48\linewidth}
\centering
\begin{tikzpicture}
  \node[graph vertex] (A) at (0,1.6) {$A$};
  \node[graph vertex] (B) at (1.9,1.6) {$B$};
  \node[graph vertex] (C) at (1.9,0) {$C$};
  \node[graph vertex] (D) at (0,0) {$D$};

  \foreach \u/\v in {A/B,B/C,C/D,D/A} {
    \draw[graph edge] (\u)
      to[bend left=15] (\v);
    \draw[graph edge] (\v)
      to[bend left=15] (\u);
  }

  \node[font=\small] at (0.95,-0.9)
    {A double unicycle};
\end{tikzpicture}

\smallskip
{\small
$\begin{gathered}
 A=\{1,5\},\quad B=\{2,8\},\\
 C=\{3,7\},\quad D=\{4,6\}.
\end{gathered}$
}
\end{minipage}

\caption{Two quotients $T^\pi$ of $T_8$, with
$\pi=\{A,B,C,D\}$ as specified in each panel.
On the left, the underlying simple graph is a tree:
the edge $AB$ has multiplicity four and the other edges
have multiplicity two.
On the right, the underlying simple graph has a unique cycle
and every edge has multiplicity two.
Both examples satisfy $\#(\pi)=4=8/2$.}
\label{fig:two-four-tree-unicycle}
\end{figure}

\begin{proposition}\label{Proposition: Characterization of pi=m/2}
Let $\pi\in \cP(m)$ be such that $\overline{T^\pi}$ has no edges of multiplicity $1$. Then $\#(\pi)=m/2$ if and only if $T^\pi$ is either a 2-4 tree or a double unicycle.
\end{proposition}
\begin{proof}
It is satisfied that 
$$\#(\pi)-\#(\overline{\pi})\leq 1,$$
and
$$\#(\overline{\pi})-m/2\leq 0.$$
So $\#(\pi)=m/2$ if and only if either $\#(\pi)-\#(\overline{\pi})=1$ and $\#(\overline{\pi})-m/2=-1$ or
$\#(\pi)-\#(\overline{\pi})=0$ and $\#(\overline{\pi})-m/2=0$. In the first case $T^\pi$ is a 2-4 tree while in the second case $T^\pi$ is a double unicycle. 
\end{proof}

\begin{remark}\label{Remark: Properties of double unicycle and 2-4 tree}
As pointed out in Remark \ref{Remark: Properties of double tree}, if $T^\pi$ is a 2-4 tree then any two edges joining the same pair of vertices have the opposite orientation except for the four edges $e_u,e_v,e_w,e_x$ that join the same pair of vertices. In this case, it must be that two of them have the same orientation, say $e_u,e_w$ and opposite to the orientation of the other two edges, $e_v$ and $e_x$. Further, let $a$ and $b$ be the two vertices of $T^\pi$ such that $e_u,e_v,e_w,e_x$ join the vertices $a$ and $b$. Starting from $a$, we must be able to traverse the entire graph $T^\pi$ by following the edges of the cycle $T$. Suppose $e_u=(a,b)$ with $u$ even, Since $T^\pi$ is a double tree, every time we traverse an edge, we must traverse the same edge again in the opposite direction, so either $v$ or $x$ are odd and therefore $w$ must be even. Thus, $u,w$ are both even while $v,x$ are both odd. If $T^\pi$ is a double unicycle then any two edges $e_u$ and $e_v$ joining the same pair of vertices are in the following way: if $\overline{e_u}$ is not in the cycle of $\overline{T^\pi}$ then $e_u$ and $e_v$ have the opposite orientation. If $\overline{e_u}$ is in the cycle of $\overline{T^\pi}$ then they might have or not the opposite orientation. If they have the opposite orientation any other two edges joining the same pair of vertices and within the cycle must also have the opposite orientation. In this case we say $T^\pi$ is a \textit{double unicycle of first type}. If $e_u$ and $e_v$ have the same orientation then any other two edges joining the same pair of vertices within the cycle must also have the same orientation. In this case we say $T^\pi$ is a \textit{double unicycle of second type}. Example of these graphs can be seen in Figure \ref{fig:double-unicycle-types}. The last fact is a direct consequence that any vertex of $T^\pi$ must have the same number of incoming and outgoing edges.
\end{remark}

\begin{figure}[htbp]
\centering
\paperfigurefont
\begin{minipage}{0.48\linewidth}
\centering
\begin{tikzpicture}
  \node[graph vertex] (A) at (0,1.6) {$A$};
  \node[graph vertex] (B) at (1.8,1.6) {$B$};
  \node[graph vertex] (C) at (1.8,0) {$C$};
  \node[graph vertex] (D) at (0,0) {$D$};
  \node[graph vertex] (E) at (-1.5,1.6) {$E$};

  \foreach \u/\v in {A/B,B/C,C/D,D/A,A/E} {
    \draw[graph edge] (\u)
      to[bend left=15] (\v);
    \draw[graph edge] (\v)
      to[bend left=15] (\u);
  }

  \node[font=\small] at (0.15,-0.75)
    {First type};
\end{tikzpicture}
\end{minipage}\hfill
\begin{minipage}{0.48\linewidth}
\centering
\begin{tikzpicture}
  \node[graph vertex] (A) at (0,1.6) {$A$};
  \node[graph vertex] (B) at (1.8,1.6) {$B$};
  \node[graph vertex] (C) at (1.8,0) {$C$};
  \node[graph vertex] (D) at (0,0) {$D$};
  \node[graph vertex] (E) at (-1.5,1.6) {$E$};

  \foreach \u/\v in {B/A,C/B,D/C,A/D} {
    \draw[graph edge] (\u)
      to[bend left=15] (\v);
    \draw[graph edge] (\u)
      to[bend right=15] (\v);
  }

  \draw[graph edge] (A) to[bend left=15] (E);
  \draw[graph edge] (E) to[bend left=15] (A);

  \node[font=\small] at (0.15,-0.75)
    {Second type};
\end{tikzpicture}
\end{minipage}

\caption{Double unicycles of the first and second types.
For the first type, the two edges over every cycle edge
have opposite orientations.
For the second type, they have the same orientation
and follow a consistent direction around the cycle.
The two edges over the off-cycle edge $AE$ have opposite
orientations in both cases.
Each graph has ten directed edges.}
\label{fig:double-unicycle-types}
\end{figure}
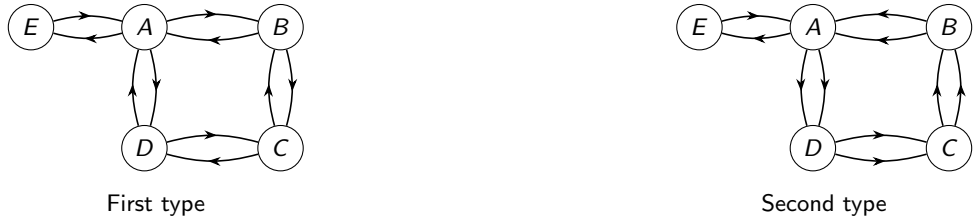

Note that if $T^\pi$ is a double tree and we merge two vertices of $T^\pi$ that are at distance $2$, the resulting graph is a 2-4 tree. We would like to prove that this is actually the only way we can get a 2-4 tree.

\begin{notation}\label{Notation: New pi of 2-4 tree}
Let $\pi\in \cP(m)$ be such that $T^\pi$ is a 2-4 tree. Let 
$$B=\{e_u,e_v,e_w,e_x\}$$ 
be the edge of $\overline{T^\pi}$ of multiplicity $4$ with $e_u,e_w$ having the same orientation and opposite to $e_v,e_x$. Let $\overline{\pi}^\prime\in \cP_2(m)$ be the pairing whose blocks are the blocks of $\overline{\pi}$ except for $B$ which turns into the blocks $\{u,v\}$ and $\{w,x\}$ of $\overline{\pi}^\prime$.
\end{notation}

\begin{remark}
Observe that the choice of $\overline{\pi}^\prime$ in Notation \ref{Notation: New pi of 2-4 tree} is not unique. The blocks of size $2$ of $\overline{\pi}$ determine a unique block of $\overline{\pi}^\prime$. However, the block of size $4$, $\{u,v,w,x\}$, of $\overline{\pi}$ turns into the blocks $\{u,v\}$ and $\{w,x\}$ of $\overline{\pi}^\prime$. This choice is not unique, for instance $\{u,x\}$ and $\{v,w\}$ is another possible choice. These are the only two possible choices that satisfy that if $\{a,b\}$ is a block of $\overline{\pi}^\prime$ then $e_a$ and $e_b$ join the same pair of vertices of $T^\pi$ with the opposite orientation.
\end{remark}

\begin{proposition}\label{Proposition: pi prime is non-crossing}
Let $\pi\in \cP(m)$ be such that $T^\pi$ is a 2-4 tree and let $\overline{\pi}^\prime$ be as in Notation \ref{Notation: New pi of 2-4 tree}. Then $\overline{\pi}^\prime\in \NC_2(m)$.
\end{proposition}
\begin{proof}
We prove this by induction on $m$. The base case is $m=4$. The only graph $T^\pi$ that is a 2-4 tree is given by $\pi=\{1,3\}\{2,4\}$ that has the edges $e_1,e_2,e_3,e_4$ joining the same pair of vertices with $e_1,e_3$ having the same orientation and opposite to $e_2,e_4$. Thus, $\overline{\pi}^\prime$ is either $(1,4)(2,3)$ or $(1,2)(3,4)$. In either case $\overline{\pi}^\prime\in \NC_2(4)$. Let us assume it is true for $m$ and prove it for $m+2$. Since $\overline{T^\pi}$ is a tree it has a leaf $B$. The vertex $B$ of the graph $T^\pi$ must be adjacent to exactly two edges: $e_{u}$ and $e_{u-1}$, so that $B$ is of the form $B=\{u\}$. Let $A$ be the other vertex adjacent to the edge $e_{u}$, so $u+1,u-1\in A$. Let $\hat{\pi}\in \cP(\{1,\dots,u-1,u+2,\dots,m+2\})$ be the partition whose blocks are the blocks of $\pi$ except $B$ which is removed and $A$ which now becomes $\hat{A}=A\setminus \{u+1\}$. Let $\hat{T}$ be the cycle obtained from removing $e_{u-1}$ and $e_u$ of $T$ and merging the vertices $u-1$ and $u+1$ which now identify as the vertex $u-1$. Note that $\hat{T}^{\hat{\pi}}$ is the graph obtained from removing the vertex $B$ of $T^\pi$ and the edges $e_{u-1}$ and $e_u$. Thus $\hat{T}^{\hat{\pi}}$ is a 2-4 tree and hence by induction hypothesis $\overline{\hat{\pi}}^\prime$ is non-crossing pairing. However this partition is precisely $\overline{\pi}^\prime \setminus \{u-1,u\}$, thus $\overline{\pi}^\prime$ must be non-crossing pairing.
\end{proof}

The following lemma characterizes a 2-4 tree as it proves the only way to get a 2-4 tree is merging two vertices (that are at distance 2) of a double tree.

\begin{lemma}\label{Lemma: Characterization of 2-4 tree}
Let $\pi\in \cP(m)$. Then $T^\pi$ is a 2-4 tree if and only if $T^\pi$ is obtained by identifying two vertices of $T^\sigma$ that are at distance $2$, where $T^\sigma$ is a double tree.
\end{lemma}
\begin{proof}
The converse is clear so let us prove that if $T^\pi$ is a 2-4 tree, then $T^\pi$ results of joining two vertices of $T^\sigma$ that are at distance $2$, where $T^\sigma$ is a double tree. Let $\sigma=\gamma\overline{\sigma}$ with $\overline{\sigma}=\overline{\pi}^\prime\in \NC_2(m)$ (thanks to Proposition \ref{Proposition: pi prime is non-crossing}). By Lemmas \ref{Lemma: First characterization of pi} and \ref{Lemma: Second characterization of pi} we know $T^\sigma$ is a double tree. Since $\overline{\sigma}\in \NC_2(m)$
$$\#(\sigma)+\#(\overline{\sigma})=\#(\gamma\overline{\sigma})+\#(\overline{\sigma})=m+1.$$
So we conclude $\#(\sigma)=m/2+1=\#(\pi)+1$. However, Lemma \ref{Lemma: Gamma sigma overline less or equal than pi} says
$$\sigma \leq \pi,$$
thus $\pi$ must be the result of merging two blocks of $\sigma$. These blocks must correspond to vertices of $T^\sigma$ that are at distance 2 so that the resulting graph $T^\pi$ is a 2-4 tree.
\end{proof}

Similarly as before, if $T^\pi$ is a double tree and now we instead merge two vertices of $T^\pi$ that are at distance distinct from $2$, the resulting graph is a double unicycle of first type. We would like to prove that this is the only way we can get double unicycle of first type.

\begin{proposition}\label{Proposition: pi is non-crossing}
Let $\pi\in \cP(m)$ be such that $T^\pi$ is a double unicycle of first type. Then $\overline{\pi}\in \NC_2(m)$.
\end{proposition}
\begin{proof}
We first prove it for the case when $\overline{T^\pi}$ has no leaves, that is, it is a cycle. 

Starting at $1$ and following the edges $e_m,e_{m-1},\dots,e_1$ in this order, we can travel around $T$, so we must be able to do the same in $T^\pi$. $e_m$ goes from $[1]_\pi$ to $[m]_\pi$, then there are two cases either $e_{m-1}$ goes back to $[1]_\pi$ or continues to another vertex $[m-1]_\pi$ distinct of $[1]_\pi$. The same occurs with $e_{m-2}$, at some point we must go back, let $e_s$ be this instant so that $e_s$ goes from $[s+1]_\pi$ to $[s]_\pi$ and $e_{s-1}$ goes from $[s]_\pi$ back to $[s+1]_\pi$, i.e. $[s+1]_\pi=[s-1]_\pi$. Let $\hat{\pi}$ be the partition with the same blocks as $\pi$ except the block $[s]_\pi$ which becomes the two blocks $[s]_\pi\setminus\{s\}$ and $\{s\}$. Therefore $T^{\hat{\pi}}$ is a double tree and hence by Lemmas \ref{Lemma: First characterization of pi} and \ref{Lemma: Second characterization of pi} it follows that $\overline{\hat{\pi}}\in \NC_2(m)$. However $\overline{\pi}=\overline{\hat{\pi}}$ because $T^{\hat{\pi}}$ results of splitting the vertex $[s]_\pi$ of $T^\pi$ into two vertices, turning the cycle of $\overline{T^\pi}$ into a tree but leaving the edges joining exactly the same pair of vertices. Thus $\overline{\pi}\in \NC_2(m)$. Now we prove it for arbitrary $\overline{T^\pi}$. We do this by induction on $m$. The base case is $m=6$ which follows from the preceding argument because in this case $\overline{T^\pi}$ has no leaves. Let us assume it is true for $m$ and prove it for $m+2$. We can assume $\overline{T^\pi}$ has a leaf as the other case was already proved. We proceed as in Proposition \ref{Proposition: pi prime is non-crossing}. The graph $\overline{T^\pi}$ has a leaf $[u]_\pi=\{u\}$. Then $e_u$ and $e_{u-1}$ are adjacent to $[u]_\pi$ and to other vertex $A$ so that $u-1,u+1\in A$. The rest of the proof follows exactly as in Proposition \ref{Proposition: pi prime is non-crossing}.
\end{proof}

\begin{lemma}\label{Lemma: Characterization of double unicycle first type}
Let $\pi\in \cP(m)$. Then $T^\pi$ is a double unicycle of first type if and only if $T^\pi$ results of joining two vertices of $T^\sigma$ that are at distance distinct from $2$, where $T^\sigma$ is a double tree.
\end{lemma}
\begin{proof}
The converse is clear so let us prove the other implication. Let $\sigma=\gamma\overline{\sigma}$ with $\overline{\sigma}=\overline{\pi}\in \NC_2(m)$ (thanks to Proposition \ref{Proposition: pi is non-crossing}). As in Lemma \ref{Lemma: Characterization of 2-4 tree} we conclude $T^\sigma$ is a double tree. Since $\overline{\sigma}\in \NC_2(m)$
$$\#(\sigma)+\#(\overline{\sigma})=\#(\gamma\overline{\sigma})+\#(\overline{\sigma})=m+1.$$
So we conclude $\#(\sigma)=m/2+1=\#(\pi)+1$. However, Lemma \ref{Lemma: Gamma sigma overline less or equal than pi} says
$$\sigma \leq \pi,$$
thus $\pi$ must be the result of merging two blocks of $\sigma$. These blocks must correspond to vertices of $T^\sigma$ that are at distance distinct from 2 so that the resulting graph is a double unicycle.
\end{proof}

Finally, we do not characterize double unicycles of the second type, as they exhibit more chaotic behavior and are not necessary for our purposes.

\section{Asymptotic expansion of the moments}\label{Section: Asymptotic expansion of the moments}

Throughout the paper we will verify asymptotic freeness and infinitesimal freeness of independent covariance matrices. Thus we are interested in the $n$-power expansion of the term

$$\frac{1}{n} \E \circ \Tr(M^{(r_1)}\cdots M^{(r_q)}).$$

In this section we develop a first expansion that later will help us to derive a $1/n$ expansion of the moments. Let us compute our first expansion.

\begin{proposition}\label{Proposition: Expansion of moments 1}
Let $M^{(r_1)},\dots,M^{(r_q)}$ be a collection of covariance matrices, then
\begin{equation}
\frac{1}{n} \E \circ \textup{\Tr}(M^{(r_1)}\cdots M^{(r_q)})=\frac{1}{n^{q+1}}\sum_{\pi\in \cP(2q)}\sum_{\substack{i_{odd}=1,\dots,n \\ i_{even}=1,\dots,p \\ ker(i)=\pi}}\E(x_i),
\end{equation}
where,
$$\E(x_i) := \E\Bigg(\prod_{k=1,3,\dots,2q-1}x_{i_k,i_{k+1}}^{(r_{\frac{k+1}{2}})}\overline{x_{i_{k+2},i_{k+1}}^{(r_{\frac{k+1}{2}})}}\Bigg).$$
Here $i_{odd}=i_1,i_3,\dots,i_{2q-1}$ and $i_{even}=i_2,i_4,\dots,i_{2q}$.
\end{proposition}
\begin{proof}
\begin{eqnarray*}
&&\frac{1}{n} \E \circ \Tr(M^{(r_1)}\cdots M^{(r_q)}) \\
&=& \frac{1}{n} \sum_{i_1,\dots ,i_q=1}^n \E(M^{(r_1)}_{i_1,i_2}M^{(r_2)}_{i_2,i_3}\cdots M^{(r_q)}_{i_q,i_1}) \\
&=& \frac{1}{n^{q+1}}\sum_{i_1,\dots,i_q=1}^n \sum_{l
_1,\dots,l_q=1}^p \E(x_{i_1,l_1}^{(r_1)}\overline{x_{i_2,l_1}^{(r_1)}}x_{i_2,l_2}^{(r_2)}\overline{x_{i_3,l_2}^{(r_2)}}\cdots x_{i_q,l_q}^{(r_q)}\overline{x_{i_1,l_q}^{(r_q)}}) \\
&=& \frac{1}{n^{q+1}}\sum_{i_1,i_3\dots,i_{2q-1}=1}^n \sum_{i_2,i_4,\cdots,i_{2q}=1}^p \E(x_{i_1,i_2}^{(r_1)}\overline{x_{i_3,i_2}^{(r_1)}}x_{i_3,i_4}^{(r_2)}\overline{x_{i_5,i_4}^{(r_2)}}\cdots x_{i_{2q-1},i_{2q}}^{(r_q)}\overline{x_{i_1,i_{2q}}^{(r_q)}}) \\
&=& \frac{1}{n^{q+1}}\sum_{\pi\in \cP(2q)}\sum_{\substack{i_{odd}=1,\dots,n \\ i_{even}=1,\dots,p \\ ker(i)=\pi}}\E\Bigg(\prod_{k=1,3,\dots,2q-1}x_{i_k,i_{k+1}}^{(r_{\frac{k+1}{2}})}\overline{x_{i_{k+2},i_{k+1}}^{(r_{\frac{k+1}{2}})}}\Bigg).
\end{eqnarray*}
\end{proof}

In order to get a $n$-power expansion, we require to know the $n$-order of each of the terms in the summation for every fixed $\pi\in \cP(2q)$. This is established in the following proposition.

\begin{proposition}\label{Proposition: N order of pi sum}
Let $\pi\in \cP(2q)$, then
$$\Bigg| \sum_{\substack{i_{odd}=1,\dots,n \\ i_{even}=1,\dots,p \\ ker(i)=\pi}}\E(x_i) \Bigg| =O(n^{\#(\pi)}).$$
\end{proposition}
\begin{proof}
First of all note that 
$$|\E(x_i)|\leq \E\Bigg(\prod_{k=1,3,\dots,2q-1}|x_{i_k,i_{k+1}}^{(r_{\frac{k+1}{2}})}||\overline{x_{i_{k+2},i_{k+1}}^{(r_{\frac{k+1}{2}})}}|\Bigg),$$
where the right hand side only depends on $ker(i)$ because $|x_{i,j}|=|\overline{x_{i,j}}|$ so we are reduced to prove that 
$$\Bigg| \sum_{\substack{i_{odd}=1,\dots,n \\ i_{even}=1,\dots,p \\ ker(i)=\pi}}1 \Bigg| =O(n^{\#(\pi)}).$$
We prove this by induction on $\#(\pi)$. The case $\#(\pi)=1$ follows directly as
$$\Bigg| \sum_{\substack{i_{odd}=1,\dots,n \\ i_{even}=1,\dots,p \\ ker(i)=\pi}}1 \Bigg|=\Bigg| \sum_{\substack{i_{odd}=1,\dots,n \\i_{even}=1,\dots,p \\ i_{1}=i_2=\cdots =i_{2q}}}1 \Bigg|=\min\{n,p\} =O(n).$$
Now we assume it is true for $\#(\pi)\leq m$ and prove it for $\#(\pi)=m+1$. 
Let 
$$\mathcal{E}=\{B\in\pi: B\subset \{2,4,\dots,2q\}\},$$
$$\mathcal{O}=\{B\in\pi: B\subset \{1,3,\dots,2q-1\}\},$$
and
$$\mathcal{M}=\{B\in\pi: B\cap \{1,\dots,2q-1\}\neq \emptyset,\ \text{and}\ B\cap \{2,\dots,2q\}\neq \emptyset\}.$$
Let
\begin{multline*}
L=\{\rho\in\cP(2q):\rho\neq\pi \text{ and every block of }\rho=A\cup B\cup C,
A\in \mathcal{E}\cup \{\emptyset\}, B\in \mathcal{O}\cup\{\emptyset\},C\in\mathcal{M}\cup\{\emptyset\}\}
\end{multline*}
Note that every partition $\rho\in L$ is such that $\pi\leq \rho$ because $\rho$ merges blocks of $\pi$. Further, $\rho$ never merges blocks of $\pi$ in the same set $\mathcal{E},\mathcal{O}$ or $\mathcal{M}$. 
Observe that $(p)_{|\mathcal{E}|}$ counts the number of ways of choosing $|\mathcal{E}|$ distinct indices from $1$ to $p$. Here $(n)_k=n(n-1)\cdots (n-k+1)$ is the falling factorial. Therefore
$$(p)_{|\mathcal{E}|}(n)_{|\mathcal{O}|}(\min\{n,p\})_{|\mathcal{M}|},$$
represents the number of ways of choosing $|\mathcal{E}|$ distinct indices from $1,\dots,p$, $|\mathcal{O}|$ distinct indices from $1,\dots,n$ and $|\mathcal{M}|$ distinct indices from $1,\dots,\min\{n,p\}$. This is almost the number of ways of choosing indices $i_1,\dots,i_{2q}$ such that $ker(i)=\pi$ however, we are allowing repetition of the indices between the sets $\mathcal{E},\mathcal{O}$ and $\mathcal{M}$. Thus,
$$\sum_{\substack{i_{odd}=1,\dots,n \\ i_{even}=1,\dots,p \\ ker(i)=\pi}}1=(p)_{|\mathcal{E}|}(n)_{|\mathcal{O}|}(\min\{n,p\})_{|\mathcal{M}|}-\sum_{\rho\in L}\sum_{\substack{i_{odd}=1,\dots,n \\ i_{even}=1,\dots,p \\ ker(i)=\rho}}1,$$
because we are precisely subtracting the case where we repeat some index.
The order of $(p)_{|\mathcal{E}|}(n)_{|\mathcal{O}|}(\min\{n,p\})_{|\mathcal{M}|}$ is $n^{|\mathcal{E}|+|\mathcal{O}|+|\mathcal{M}|}=n^{\#(\pi)}=n^{m+1}$ while by induction, for each $\rho\in L$ the order of $\sum_{\substack{i_{odd}=1,\dots,n \\ i_{even}=1,\dots,p \\ ker(i)=\rho}}1$ is $n^{\#(\rho)}$ because $\#(\rho)<\#(\pi)=m+1$. Therefore we conclude that the order of $\sum_{\substack{i_{odd}=1,\dots,n \\ i_{even}=1,\dots,p \\ ker(i)=\pi}}1$ is $n^{m+1}$ as required.
\end{proof}

We now invoke our graph theory. Let us introduce the following graph.

\begin{definition}\label{Definition: T graph of Wishart}
Let $C=(V,E)$ be the oriented graph with set of vertices $V=[2q]$ and edges
$$e_{2i-1}=(2i-1,2i), 1\leq i\leq q,$$
$$e_{2i}=(2i+1,2i), 1\leq i\leq q,$$
with the convention $2q+1=1$.
\end{definition}

\begin{remark}\label{Remark: Expectation depending on the graph}
For a choice of indices $i=(i_1,\ldots,i_{2q})$, we have
\[
\mathbb{E}(x_i)
=
\mathbb{E}\Bigg(
    \prod_{k=1,3,\ldots,2q-1}
    x_{i_{\operatorname{sr}(e_k)},
       i_{\operatorname{tr}(e_k)}}^{(r_{(k+1)/2})}
    \prod_{k=2,4,\ldots,2q}
    \overline{
        x_{i_{\operatorname{sr}(e_k)},
           i_{\operatorname{tr}(e_k)}}^{(r_{k/2})}
    }
\Bigg).
\]
\end{remark}

\begin{proposition}\label{Proposition: Graph must have edges mult at least two}
Let $\pi\in \cP(2q)$. If $\overline{C^\pi}$ has an edge of multiplicity $1$ then $\E(x_i)=0$ for any $i$ such that $ker(i)=\pi$.
\end{proposition}
\begin{proof}
If $\overline{C^\pi}$ has an edge of multiplicity $1$ then there exist vertices $u,u+1$ connected by a unique edge $(u,u+1)$ or $(u+1,u)$. Assume without loss of generality $e_u=(u,u+1)$ with $u$ odd. If $i$ is such that $ker(i)=\pi$ by Remark \ref{Remark: Expectation depending on the graph}, $\E(x_i)$ has a factor of the form $\E(x_{i_{sr(e_u)},i_{tr(e_u)}}^{(r_{\frac{u+1}{2}})})=0$ where this variable appears exactly once because the edge has multiplicity $1$.
\end{proof}

\begin{lemma}\label{Lemma: Expansion of moments 2}
Let $M^{(r_1)},\dots,M^{(r_q)}$ be a collection of covariance matrices, then
$$\frac{1}{n}\E\circ \textup{\Tr}(M^{(r_1)}\cdots M^{(r_q)})=\frac{1}{n^{q+1}}\sum_{\substack{\pi\in \cP(2q)\\ \overline{C^\pi} \text{ has no edges} \\\text{of multiplicity }1}}\sum_{\substack{i_{odd}=1,\dots,n \\ i_{even}=1,\dots,p \\ ker(i)=\pi}}\E(x_i)$$
\end{lemma}
\begin{proof}
The proof follows immediately by combining Propositions \ref{Proposition: Expansion of moments 1} and \ref{Proposition: Graph must have edges mult at least two}.
\end{proof}

Lemma \ref{Lemma: Expansion of moments 2} gives an expansion of the moments in terms of the graphs $C^\pi$. In Sections \ref{Section: Moments} and \ref{Section: Infinitesimal moments} we further develop an expansion of the moments in terms of the powers of $\frac{1}{n}$. 

\section{Free moments, cumulants and asymptotic freeness}\label{Section: Moments}

In Lemma \ref{Lemma: Expansion of moments 2} we got an expansion of the moments, in this section we would like to further find the explicit coefficient associated to the constant term. Now we consider independent covariance matrices as we would like to prove asymptotic freeness. 

\begin{notation}\label{Notation: Ker rr}
For a choice of indices $r_1,\dots,r_q$ we let $ker(rr)\in \cP(2q)$ be the partition where $2u,2v,2u-1$ and $2v-1$ are in the same block of $ker(rr)$ if and only if $u$ and $v$ are in the same block of $ker(r)$.
\end{notation}

The next proposition establishes a relation between $ker(r)$ and $ker(rr)$.

\begin{proposition}\label{Proposition: Equivalence of ker(r) and Ker(rr)}
Let $\sigma\in \NC(\{2,4,\dots,2q\})$ and $Kr(\sigma)\in \NC(\{1,3,\dots,2q-1\})$ be its kreweras complement. Let $\pi=\sigma\cup Kr(\sigma)$ and let $\overline{\pi}$ be such that $\pi=\gamma\overline{\pi}$, with $\gamma=(1,\dots,2q)$. Let $r_1,\dots,r_q$ be a choice of indices. Then $\overline{\pi}\leq ker(rr)$ if and only if $\sigma/2 \leq ker(r)$. Here $\sigma/2$ is the partition of $\NC(q)$ obtained by dividing every element by two of the blocks of $\sigma$.
\end{proposition}
\begin{proof}
First assume $\overline{\pi}\leq ker(rr)$. Let $u/2$ and $v/2$ be two numbers in the same block of $\sigma/2$. Then $u$ and $v$ are both even numbers in the same block of $\sigma$. Without loss of generality assume $u=\sigma(v)$, then
$$\gamma\overline{\pi}(v)=\pi(v)=\sigma(v)=u,$$
thus $\overline{\pi}(v)=\gamma^{-1}(u)$ and therefore $v$ and $\gamma^{-1}(u)$ are in the same block of $\overline{\pi}$. By hypothesis $v$ and $\gamma^{-1}(u)$ are in the same block of $ker(rr)$ and then so are $v$ and $u$. Hence $u/2$ and $v/2$ are in the same block of $ker(r)$. Conversely suppose $\sigma/2 \leq ker(r)$. By definition $\overline{\pi}$ sends even (odd) to odd (even) numbers. Let $u$ even and $v$ odd in the same block of $\overline{\pi}$ with $v=\overline{\pi}(u)$. Thus
$$\gamma(v)=\gamma\overline{\pi}(u)=\pi(u)=\sigma(u),$$
thus $u$ and $\gamma(v)$ are in the same block of $\sigma$ and therefore $u/2$ and $\gamma(v)/2$ are in the same block of $\sigma/2$. By hypothesis $u/2$ and $\gamma(v)/2$ are in the same block of $ker(r)$ and then $u$ and $\gamma(v)$ are in the same block of $ker(rr)$, so are $u$ and $v$, as required.
\end{proof}


\begin{proposition}\label{Proposition: contribution of double tree}
Let $M^{(r_1)},\dots, M^{(r_q)}$ be a collection of independent covariance matrices. Let $\pi\in \cP(2q)$ be such that $\pi=\gamma\overline{\pi}$ with $\overline{\pi}\in \NC_2(2q)$. Let $i_1,\dots,i_{2q}$ be such that $ker(i)=\pi$. If $\E(x_i)\neq 0$, then $\overline{\pi}\leq ker(rr)$ and $\E(x_i)=1$. 
\end{proposition}
\begin{proof}

By Lemmas \ref{Lemma: First characterization of pi} and \ref{Lemma: Second characterization of pi} we know the graph $C^\pi$ is a double tree where any block $\{u,v\}\in\overline{\pi}$ is such that $e_u$ and $e_v$ join the same pair of vertices. Further $u$ and $v$ must have distinct parity as $\overline{\pi}$ is non-crossing pairing. Let $u,v$ be such that both are in the same block of $\overline{\pi}$ with $u$ even and $v$ odd. Since $e_u$ and $e_v$ connect the same pair of vertices, $x_{i_{sr(e_u)},i_{tr(e_u)}}^{(r_{\frac{u}{2}})}$ and $x_{i_{sr(e_v)},i_{tr(e_v)}}^{(r_{\frac{v+1}{2}})}$ have the same pair of indices. If we require $\E(x_i)\neq 0$ we must have $r_{\frac{u}{2}}=r_{\frac{v+1}{2}}$ so that the entries are not independent. Thus $u/2$ and $(v+1)/2$ are both in the same block of $ker(r)$ and therefore $u$ and $v$ are in the same block of $ker(rr)$. This proves $\overline{\pi}\leq ker(rr)$. Finally, thanks to Remark \ref{Remark: Properties of double tree} the edges $e_u$ and $e_v$ join the same pair of vertices with the same orientation (let us recall the orientation of odd edges in $C$ are opposite to odd edges in $T_{2q}$). From this fact we conclude that $\E(x_i)$ has all factors of the form $\E(x_{i,j}\overline{x_{i,j}})=\E(|x_{i,j}|^2)=1$. Hence $\E(x_i)=1$.
\end{proof}

We are ready to prove our main results on the moments.

\begin{proof}[Proof of Theorem \ref{Lemma: Expansion of moments 4}]
Let us recall the Equation of Lemma \ref{Lemma: Expansion of moments 2}.
$$\frac{1}{n}\E\circ \textup{\Tr}(M^{(r_1)}\cdots M^{(r_q)})=\frac{1}{n^{q+1}}\sum_{\substack{\pi\in \cP(2q)\\ \overline{C^\pi} \text{ has no edges} \\\text{of multiplicity }1}}\sum_{\substack{i_{odd}=1,\dots,n \\ i_{even}=1,\dots,p \\ ker(i)=\pi}}\E(x_i)$$
By Proposition \ref{Proposition: N order of pi sum} we know for each $\pi\in \cP(2q)$ the sum over $i$ has order $n^{\#(\pi)}$, further, the graph $C$ and $T_{2q}$ defined in notation \ref{Notation: Graph m cycle} are the same up to orientation of the edges, therefore, thanks to Lemma \ref{Lemma: Second characterization of pi} we know
$$\#(\pi)\leq q+1,$$
with equality if and only if $\pi=\gamma\overline{\pi}$ and $\overline{\pi}\in \NC_2(2q)$. Therefore
\begin{equation*}
\frac{1}{n}\E\circ \Tr(M^{(r_1)}\cdots M^{(r_q)})=\frac{1}{n^{q+1}}\sum_{\substack{\overline{\pi}\in \NC_2(2q) \\ \pi=\gamma\overline{\pi}}}\sum_{\substack{i_{odd}=1,\dots,n \\ i_{even}=1,\dots,p \\ ker(i)=\pi}}\E(x_i)+O(\frac{1}{n}).
\end{equation*}
Here $\gamma=(1,\dots,2q)\in S_{2q}$. By Proposition \ref{Proposition: contribution of double tree} the last equation simplifies to
\begin{equation}\label{Aux 3}
\frac{1}{n}\E\circ \Tr(M^{(r_1)}\cdots M^{(r_q)})=\frac{1}{n^{q+1}}\sum_{\substack{\overline{\pi}\in \NC_2(2q) \\ \pi=\gamma\overline{\pi} \\ \overline{\pi}\leq ker(rr)}}\sum_{\substack{i_{odd}=1,\dots,n \\ i_{even}=1,\dots,p \\ ker(i)=\pi}}1+O(\frac{1}{n}).
\end{equation}
Observe that $\pi\in \NC(2q)$ because it is the kreweras complement of a non-crossing pairing. Further, any block of $\pi$ only contains odd or even numbers because $\gamma$ and $\overline{\pi}$ send odd (and even) to even (and odd) numbers. Thus $\pi$ can be written as $\sigma \cup Kr(\sigma)$ where $\sigma\in \NC(\{2,\dots,2q\})$. From Proposition \ref{Proposition: Equivalence of ker(r) and Ker(rr)} and Equation \ref{Aux 3} we get
\begin{equation}\label{Aux 4}
\frac{1}{n}\E\circ \Tr(M^{(r_1)}\cdots M^{(r_q)})=\frac{1}{n^{q+1}}\sum_{\substack{\sigma\in \NC(\{2,\dots,2q\}) \\ \pi=\sigma\cup Kr(\sigma) \\ \sigma/2\leq ker(r)}}\sum_{\substack{i_{odd}=1,\dots,n \\ i_{even}=1,\dots,p \\ ker(i)=\pi}}1+O(\frac{1}{n}).
\end{equation}
To conclude observe that
$$\sum_{\substack{i_{odd}=1,\dots,n \\ i_{even}=1,\dots,p \\ ker(i)=\pi}}1=(p)_{|\#(\sigma)|}(n)_{|\#(Kr(\sigma))|}+O(n^{\#(\sigma)+\#(Kr(\sigma))-1}),$$
the latter  follows from the same argument used in Proposition \ref{Proposition: N order of pi sum}. Combining this with equation \ref{Aux 4} yields
\begin{equation*}
\frac{1}{n}\E\circ \Tr(M^{(r_1)}\cdots M^{(r_q)})=\frac{1}{n^{q+1}}\sum_{\substack{\sigma\in \NC(\{2,\dots,2q\}) \\ \sigma/2\leq ker(r)}}(p)_{|\#(\sigma)|}(n)_{|\#(Kr(\sigma))|}+O(\frac{1}{n}),
\end{equation*}
which simplifies to
\begin{equation*}
\frac{1}{n}\E\circ \Tr(M^{(r_1)}\cdots M^{(r_q)})=\sum_{\substack{\sigma\in \NC(\{2,\dots,2q\}) \\ \sigma/2\leq ker(r)}}\left(\frac{p}{n}\right)^{\#(\sigma)}+O(\frac{1}{n}).
\end{equation*}
\end{proof}

\begin{proof}[Proof of Corollary \ref{Corollary: Moments and cumulants}]
1) Follows immediately by taking limit in Theorem \ref{Lemma: Expansion of moments 4}. The moment-cumulant relation says that the cumulants are given by 
$$\kappa_n(M^{(r_{i_1})},\dots,M^{(r_{i_n})})=c,$$
whenever $i_1 =\cdots =i_n$ and $0$ otherwise which proves 2). Finally the condition $\sigma \leq ker(r)$ is equivalent to vanishing of mixed cumulants which simplifies to freeness, proving 3).
\end{proof}

\section{Infinitesimal free moments, cumulants and asymptotic infinitesimal freeness}\label{Section: Infinitesimal moments}

In this section, we present a more detailed version of Theorem \ref{Lemma: Expansion of moments 4}. More precisely, we now compute the coefficient of order $\frac{1}{n}$. In order to make this easier for the reader, we do first the case when $\mathbb{E}(x_{i,j}^2)=0$ in Subsection \ref{Section: zero pseudo variance}. Then, we attack the general case in Subsection \ref{Section: Case of non zero psudovariance}.

\subsection{Zero pseudo variance}\label{Section: zero pseudo variance}
We set $T=T_{2q}$ as defined in Notation \ref{Notation: Graph m cycle} and $C$ as in Definition \ref{Definition: T graph of Wishart}. Recall that $T$ and $C$ are the same graph up to orientation of the edges, thus all results for $T$ also apply to $C$ except those involving the orientation of the edges.

Let us recall Equation of Lemma \ref{Lemma: Expansion of moments 2}
\begin{equation}\label{Aux 7}
\frac{1}{n}\E\circ \Tr(M^{(r_1)}\cdots M^{(r_q)})=\frac{1}{n^{q+1}}\sum_{\substack{\pi\in \cP(2q)\\ \overline{C^\pi} \text{ has no edges} \\\text{of multiplicity }1}}\sum_{\substack{i_{odd}=1,\dots,n \\ i_{even}=1,\dots,p \\ ker(i)=\pi}}\E(x_i).
\end{equation}
From Proposition \ref{Proposition: N order of pi sum} we know each sum $\sum_{\substack{i_{odd}=1,\dots,n \\ i_{even}=1,\dots,p \\ ker(i)=\pi}}\E(x_i)$ has order $n^{\#(\pi)}$. So, if we want to compute the coefficient of order $\frac{1}{n}$, we need to consider partitions $\pi\in\cP(2q)$ such that $\#(\pi)=q+1$ and $\#(\pi)=q$. Since $\overline{C^\pi}$ has an edge of multiplicity $1$ if and only if $\overline{T^\pi}$ does, we can consider $T^\pi$ instead of $C^\pi$. We characterized in Section \ref{Section: Graph theory} the graphs $T^\pi$ with no edges of multiplicity $1$ such that $\#(\pi)=q+1$ and $\#(\pi)=q$. These are double trees, 2-4 trees and double unicycles respectively. So, now let us compute the corresponding contribution $\E(x_i)$ for each $i$ such that $\ker(i)=\pi$ when $\#(\pi)=q$. Let us recall that the case when $\#(\pi)=q+1$ was already done in Section \ref{Section: Moments}.

\begin{proposition}\label{Proposition: contribution of 2-4 tree}
Let $M^{(r_1)},\dots, M^{(r_q)}$ be a collection of independent covariance matrices. Let $\pi\in \cP(2q)$ be such that $C^\pi$ is a 2-4 tree. Let $i_1,\dots,i_{2q}$ be such that $ker(i)=\pi$. If $\E(x_i)\neq 0$, then either
\begin{enumerate}
    \item $\overline{\pi}^\prime \leq ker(rr)$ and $\E(x_i)=1$, further only one of the choices for $\overline{\pi}^\prime$ satisfies $\overline{\pi}^\prime \leq ker(rr)$. Here $\overline{\pi}^\prime$ is as in Notation \ref{Notation: New pi of 2-4 tree}, or,
    \item $\overline{\pi}\leq ker(rr)$ and $\E(x_i)=\E|x_{1,2}|^4$. 
\end{enumerate}
\end{proposition}
\begin{proof}
Let $\{u,v\}$ be a block of $\overline{\pi}$ which is also a block of $\overline{\pi}^\prime$. By Proposition \ref{Proposition: pi prime is non-crossing} we know $\overline{\pi}^\prime\in \NC_2(2q)$ and therefore $u$ and $v$ have distinct parity, say $u$ is even and $v$ is odd. Since $e_u$ and $e_v$ connect the same pair of vertices and with the same orientation, $x_{i_{sr(e_u)},i_{tr(e_u)}}^{(r_{\frac{u}{2}})}$ and $x_{i_{sr(e_v)},i_{tr(e_v)}}^{(r_{\frac{v+1}{2}})}$ have the same pair of indices. If we require $\E(x_i)\neq 0$ we must have $r_{\frac{u}{2}}=r_{\frac{v+1}{2}}$ so that the entries are not independent. Thus $u/2$ and $(v+1)/2$ are both in the same block of $ker(r)$ and therefore $u$ and $v$ are in the same block of $ker(rr)$. This proves any block of $\overline{\pi}$ of size $2$ is contained in a block of $ker(rr)$. Further, these blocks contribute a factor of the form $\E(x_{i,j}\overline{x_{i,j}})=\E|x_{i,j}|^2=1$ to $\E(x_i)$.

Let $\{u,v,w,x\}$ be the block of $\overline{\pi}$ of size $4$. We know $T^\pi$ is also a 2-4 tree, let us consider for now $e_i$ be the edges of $T$. By Remark \ref{Remark: Properties of double unicycle and 2-4 tree} we can assume $e_u,e_w$ have the same orientation and are both even while $e_v,e_x$ have the opposite orientation to $e_u$ and are both odd. We know $\overline{\pi}^\prime$ has blocks either $\{u,v\}\{w,x\}$ or $\{u,x\}\{v,w\}$. Since $e_u,e_v,e_w,e_x$ join the same pair of vertices, $\E(x_i)$ has a factor of the form
\begin{equation}\label{Aux 5}
\E\left(x_{i_{sr(e_v)},i_{tr(e_v)}}^{(r_{\frac{v+1}{2}})}\overline{x_{i_{sr(e_u)},i_{tr(e_u)}}^{(r_{\frac{u}{2}})}}x_{i_{sr(e_x)},i_{tr(e_x)}}^{(r_{\frac{x+1}{2}})}\overline{x_{i_{sr(e_w)},i_{tr(e_w)}}^{(r_{\frac{w}{2}})}}\right)
\end{equation}
If we now instead consider the edges of $C^\pi$, it only changes the orientation of the odd edges, so that now $e_u,e_v,e_w,e_x$ have all the same orientation and therefore (\ref{Aux 5}) becomes
\begin{equation}\label{Aux 6}
\E(x_{1,2}^{(r_{\frac{v+1}{2}})}\overline{x_{1,2}^{(r_{\frac{u}{2}})}}x_{1,2}^{(r_{\frac{x+1}{2}})}\overline{x_{1,2}^{(r_{\frac{w}{2}})}}).
\end{equation}
If we don't want this to become $0$ we either have
\begin{enumerate}
    \item $r_{\frac{v+1}{2}}=r_{\frac{u}{2}}=r_{\frac{x+1}{2}}=r_{\frac{w}{2}}$, or,
    \item $r_{\frac{v+1}{2}}=r_{\frac{u}{2}}$ and $r_{\frac{x+1}{2}}=r_{\frac{w}{2}}$ with $r_{\frac{v+1}{2}}\neq r_{\frac{x+1}{2}}$, or,
    \item $r_{\frac{v+1}{2}}=r_{\frac{w}{2}}$ and $r_{\frac{x+1}{2}}=r_{\frac{u}{2}}$ with $r_{\frac{v+1}{2}}\neq r_{\frac{x+1}{2}}$.
\end{enumerate}
In the first case we get $\frac{v+1}{2},\frac{u}{2},\frac{x+1}{2},\frac{w}{2}$ are all in the same block of $ker(r)$ and therefore $u,w,v,x$ are all in the same block of $ker(rr)$. Hence $\overline{\pi}\leq ker(rr)$, further in this case (\ref{Aux 6}) becomes $\E|x_{1,2}|^4$ so that $\E(x_i)=\E|x_{1,2}|^4$. This is precisely statement 2. Similarly, in the second case we get $u,v$ are in the same block of $ker(rr)$ while $w,x$ are also in the same block of $ker(rr)$ which is distinct to the block that contains $u,v$. This means that the choice $\overline{\pi}^\prime$ with blocks $\{u,v\}\{w,x\}$ is such that $\overline{\pi}^\prime \leq ker(rr)$. Moreover, the other choice of $\overline{\pi}^\prime$ with blocks $\{u,x\}\{v,w\}$ doesn't satisfy $\overline{\pi}^\prime \leq ker(rr)$. Finally in this case (\ref{Aux 6}) becomes $1$ so that $\E(x_i)=1$. This is precisely statement 1. The third case follows the same as previous, getting again the conclusion of statement 1.
\end{proof}

\begin{proposition}\label{Proposition: contribution of double unicycle of first type}
Let $M^{(r_1)},\dots, M^{(r_q)}$ be a collection of independent covariance matrices. Let $\pi\in \cP(2q)$ be such that $T^\pi$ is a double unicycle of first type. Let $i_1,\dots,i_{2q}$ be such that $ker(i)=\pi$. If $\E(x_i)\neq 0$, then $\overline{\pi}\leq ker(rr)$ and $\E(x_i)=1$.
\end{proposition}
\begin{proof}
Let $\{u,v\}$ be a block of $\overline{\pi}$, we assume $u$ is even and $v$ is odd as $\overline{\pi}\in \NC_2(2q)$. We know $e_u$ and $e_v$ join the same pair of vertices with opposite orientation and therefore these edges have the same orientation in $C^\pi$, so that $\E(x_i)$ have a factor of the form
$$\E(x_{i_{sr(e_v)},i_{tr(e_v)}}^{(r_{\frac{v+1}{2}})}\overline{x_{i_{sr(e_u)},i_{tr(e_u)}}^{(r_{\frac{u}{2}})}})=\E(x_{1,2}^{(r_{\frac{v+1}{2}})}\overline{x_{1,2}^{(r_{\frac{u}{2}})}}).$$
If we don't want this to become $0$ we must have $r_{\frac{v+1}{2}}=r_{\frac{u}{2}}$ so that $u,v$ are in the same block of $ker(rr)$, this proves $\overline{\pi}\leq ker(rr)$. Further, $\E(x_i)$ has only factors of the form $\E(x_{1,2}\overline{x_{1,2}})=1$ so that $\E(x_i)=1$. 
\end{proof}

\begin{proposition}\label{Proposition: contribution of double unicycle of second type}
Let $M^{(r_1)},\dots, M^{(r_q)}$ be a collection of independent covariance matrices such that the entries further satisfy $\E(x_{1,2}^2)=0$. Let $\pi\in \cP(2q)$ be such that $T^\pi$ is a double unicycle of second type. Let $i_1,\dots,i_{2q}$ be such that $ker(i)=\pi$. Then $\E(x_i)= 0$.
\end{proposition}
\begin{proof}
Let $\{u,v\}\in\overline{\pi}$ be a block such that $\overline{e_u}$ is an edge of the cycle of $\overline{T^\pi}$. In this case either $u$ and $v$ have the same parity or distinct parity. If they have distinct parity then the edges $e_u$ and $e_v$ of $C^\pi$ have now the opposite orientation (they have the same orientation in $T^\pi$ and odd edges have opposite orientation in $T$ to odd edges in $C$). Thus $\E(x_i)$ has a factor of the form $\E(x_{1,2}\overline{x_{2,1}})=\E(x_{1,2})\E(\overline{x_{2,1}})=0$ by independence (it doesn't matter if variables come from the same covariance matrix, they are still independent, so we omit the superscript). Thus $\E(x_i)=0$. If they have the same parity and $u$ and $v$ are both even then the adjacent edges within the cycle of $T^\pi$ must be both odd. So let us assume $u$ and $v$ are both odd. In this case the corresponding edges $e_u$ and $e_v$ of $C^\pi$ have still the same orientation and therefore $\E(x_i)$ has a factor of the form
$$\E(x_{1,2}^{(r_{\frac{u+1}{2}})}x_{1,2}^{(r_{\frac{v+1}{2}})}).$$
The last quantity is either $\E(x_{1,2}^2)=0$ or $\E(x_{1,2})\E(x_{1,2})=0$ depending on whether $r_{\frac{u+1}{2}}=r_{\frac{v+1}{2}}$. In either case we get $\E(x_i)=0$.
\end{proof}

In order to find the coefficient of order $\frac{1}{n}$, we must recover some previous work done in Section \ref{Section: Moments}. For instance, recall Equation (\ref{Aux 4})
\begin{equation*}
\frac{1}{n}\E\circ \Tr(M^{(r_1)}\cdots M^{(r_q)})=\frac{1}{n^{q+1}}\sum_{\substack{\sigma\in \NC(\{2,\dots,2q\}) \\ \pi=\sigma\cup Kr(\sigma) \\ \sigma/2\leq ker(r)}}\sum_{\substack{i_{odd}=1,\dots,n \\ i_{even}=1,\dots,p \\ ker(i)=\pi}}1+O(\frac{1}{n}).
\end{equation*}
In Section \ref{Section: Moments} we use the fact that
$$\sum_{\substack{i_{odd}=1,\dots,n \\ i_{even}=1,\dots,p \\ ker(i)=\pi}}
1=(p)_{\#(\sigma)}(n)_{\#(Kr(\sigma))}+O(n^{\#(\sigma)+\#(Kr(\sigma))-1}).$$
Now, we care about the term with order $O(n^{\#(\sigma)+\#(Kr(\sigma))-1})$. This motivates the following proposition whose proof uses the same arguments as those used in Proposition \ref{Proposition: N order of pi sum} and therefore we omit it.

\begin{proposition}\label{Proposition: Order of sum when pi=q+1}
Let $\sigma\in \NC(\{2,\dots,2q\})$ and $\pi=\sigma\cup Kr(\sigma)$. Then
$$\sum_{\substack{i_{odd}=1,\dots,n \\ i_{even}=1,\dots,p \\ ker(i)=\pi}}1=(p)_{\#(\sigma)}(n)_{\#(Kr(\sigma))}-\sum_{\rho\in L_\pi}\sum_{\substack{i_{odd}=1,\dots,n \\ i_{even}=1,\dots,p \\ ker(i)=\rho}}1+O(n^{\#(\sigma)+\#(Kr(\sigma))-2}).$$
Here 
\begin{multline*}
L_\pi=\{\rho\in \cP(2q): \pi\leq\rho, \#(\rho)=\#(\pi)-1, \text{ and every block of }\rho\text{ is a block of }\pi \\\text{ except }\text{one block which is the union of two blocks of }\pi,A\text{ and }B \\
\text{ with } A\in \sigma,B\in Kr(\sigma)\}.
\end{multline*}
\end{proposition}

\begin{proposition}\label{Proposition: Infinitesimal contribution pi=q+1}
The following equation satisfies,
\begin{multline*}
\frac{1}{n^{q+1}}\sum_{\substack{\pi\in \cP(2q)\\ \overline{C^\pi} \text{ has no edges} \\\text{of multiplicity }1 \\ \#(\pi)=q+1}}\sum_{\substack{i_{odd}=1,\dots,n \\ i_{even}=1,\dots,p \\ ker(i)=\pi}}\E(x_i)= 
\frac{1}{n^{q+1}}\sum_{\substack{\sigma\in \NC(\{2,\dots,2q\}) \\ \pi=\sigma\cup Kr(\sigma) \\ \sigma/2\leq ker(r)}}\Bigg((p)_{\#(\sigma)}(n)_{\#(Kr(\sigma))}-\sum_{\rho\in L_\pi}\sum_{\substack{i_{odd}=1,\dots,n \\ i_{even}=1,\dots,p \\ ker(i)=\rho}}1\Bigg)+O(\frac{1}{n^2}).
\end{multline*}
Here $L_\pi$ is defined as in Proposition \ref{Proposition: Order of sum when pi=q+1}.
\end{proposition}
\begin{proof}
From the same arguments used in Theorem \ref{Lemma: Expansion of moments 4} we get that
\begin{equation*}
\frac{1}{n^{q+1}}\sum_{\substack{\pi\in \cP(2q)\\ \overline{C^\pi} \text{ has no edges} \\\text{of multiplicity }1 \\ \#(\pi)=q+1}}\sum_{\substack{i_{odd}=1,\dots,n \\ i_{even}=1,\dots,p \\ ker(i)=\pi}}\E(x_i)=
\frac{1}{n^{q+1}}\sum_{\substack{\sigma\in \NC(\{2,\dots,2q\}) \\ \pi=\sigma\cup Kr(\sigma) \\ \sigma/2\leq ker(r)}}\sum_{\substack{i_{odd}=1,\dots,n \\ i_{even}=1,\dots,p \\ ker(i)=\pi}}1.
\end{equation*}
Combining last equation with Proposition \ref{Proposition: Order of sum when pi=q+1} yields the desired result.
\end{proof}

Now we do something similar in the case when $\#(\pi)=q$. To make this easier for the reader we divide this into three cases; 2-4 tree, double unicycle of first type and double unicycle of second type. Let us get started with 2-4 tree graphs.

Recall that if $\pi$ is such that $T^\pi$ is a 2-4 tree, then there exist $\rho$ such that $T^\rho$ is a double tree and $T^\pi$ results of joining two vertices of $T^\rho$ that are at distance 2. This motivates the following notation.

\begin{notation}\label{Notation: Distance in double tree}
Let $\rho\in \cP(2q)$ be such that $T^\rho$ is a double tree. 
\begin{enumerate}
    \item Given any two blocks $A,B\in \rho$ that correspond to vertices of $\overline{T^\rho}$, there is a unique path from $A$ to $B$ in $\overline{T^\rho}$. We define the distance from $A$ to $B$ as the length of this path. We denote $d_{T^\rho}(A,B)$ such a distance.
    \item Given $A,B\in \rho$ such that $d_{T^\rho}(A,B)=2$, we let $\overline{e_a}$ and $\overline{e_b}$ be the unique edges of $\overline{T^\rho}$ that correspond to the path from $A$ to $B$ in $\overline{T^\rho}$. We can suppose without loss of generality $a,b$ are both even with $e_a$ and $e_b$ being adjacent to $A$ and $B$ respectively.
    \item Given $A,B\in \rho$, we let $\rho_{A,B}\in \cP(2q)$ be the partition whose blocks are the blocks of $\rho$ except $A$ and $B$ which become the block $A\cup B$ of $\rho_{A,B}$. 

For instance one can see Example on page 6, before Theorem \ref{Lemma: Expansion of moments 5}. 
\end{enumerate}
\end{notation}

\begin{proposition}\label{Proposition: Infinitesimal contribution pi=q for 2-4 tree}
The following equation satisfies,
\begin{multline*}
\frac{1}{n^{q+1}}\sum_{\substack{\pi\in \cP(2q)\\ T^\pi \text{is a 2-4 tree} }}\sum_{\substack{i_{odd}=1,\dots,n \\ i_{even}=1,\dots,p \\ ker(i)=\pi}}\E(x_i)= 
\sum_{\substack{\sigma\in \NC(2,\dots,2q) \\ \sigma/2\leq ker(r) \\ \rho=\sigma\cup Kr(\sigma)}}\frac{1}{n}\left(\frac{p}{n}\right)^{\#(\sigma)-1}\Bigg(\sum_{\substack{A,B\in \sigma \\ d_{T^\rho}(A,B)=2}}\alpha_{a,b}+\sum_{\substack{A,B\in Kr(\sigma) \\ d_{T^\rho}(A,B)=2}}\frac{p}{n}\alpha_{a,b}\Bigg) + O(\frac{1}{n^2}).
\end{multline*}
Here $\alpha_{a,b}=\frac{1}{2}\E|x_{1,2}|^4 \I_{r_{\frac{a}{2}}=r_{\frac{b}{2}}}+\I_{r_{\frac{a}{2}}\neq r_{\frac{b}{2}}}$, where $a,b$ depend on $A,B$ and are defined as in Notation \ref{Notation: Distance in double tree}.
\end{proposition}
\begin{proof}
Let $\pi\in \cP(2q)$ be such that $T^\pi$ is a 2-4 tree. On the one hand, by Lemma \ref{Lemma: Characterization of 2-4 tree}, $T^\pi$ can be obtained from merging two vertices at distance $2$ of $T^\rho$ where $T^\rho$ is a double tree. Moreover $\rho=\gamma\overline{\rho}$ with $\overline{\rho}=\overline{\pi}^\prime \in \NC_2(2q)$. Conversely, merging two vertices at distance $2$ from a double tree produces a 2-4 tree. 
On the other hand, if $i_1,\dots,i_{2q}$ is such that $ker(i)=\pi$ and $\E(x_i)\neq 0$. From Proposition \ref{Proposition: contribution of 2-4 tree} we know either $\overline{\pi}^\prime \leq ker(rr)$ and $\E(x_i)=1$ or $\overline{\pi}\leq ker(rr)$ and $\E(x_i)=\E|x_{1,2}|^4$. In the first case we get $\overline{\rho}=\overline{\pi}^\prime \leq ker(rr)$, further, there is only one $\overline{\pi}^\prime$ that satisfies this. In the second case we get $\overline{\rho}=\overline{\pi}^\prime \leq ker(rr)$ for the two choices of $\overline{\pi}^\prime$. Thus we can rewrite the expression
$$\sum_{\substack{\pi\in \cP(2q)\\ T^\pi \text{is a 2-4 tree} }}\sum_{\substack{i_{odd}=1,\dots,n \\ i_{even}=1,\dots,p \\ ker(i)=\pi}}\E(x_i),$$
as a sum indexed by $\overline{\rho}\in \NC_2(2q)$ with $\overline{\rho}\leq ker(rr)$, $\rho=\gamma\overline{\rho}$ and $\pi=\rho_{A,B}$ for two vertices $A,B$ at distance two in $T^\rho$. We just need to consider two cases. If $e_a,e_b$ are such that $r_{\frac{a}{2}}=r_{\frac{b}{2}}$, then merging the vertices produces a 2-4 tree where $\overline{\pi}\leq ker(rr)$ and $\E(x_i)=\E|x_{1,2}|^4$. In this case there is another $\overline{\rho} \leq ker(rr)$ that also produces the same $\pi$ and hence we consider both pairings with a contribution of $\frac{1}{2}\E|x_{1,2}|^4$ each. If $e_a,e_b$ are such that $r_{\frac{a}{2}} \neq r_{\frac{b}{2}}$ then merging the vertices produces a 2-4 tree where $\overline{\pi}^\prime \leq ker(rr)$ and $\E(x_i)=1$. In this case there is only one $\overline{\rho} \leq ker(rr)$. Hence,
\begin{equation*}
\sum_{\substack{\pi\in \cP(2q)\\ T^\pi \text{is a 2-4 tree} }}\sum_{\substack{i_{odd}=1,\dots,n \\ i_{even}=1,\dots,p \\ ker(i)=\pi}}\E(x_i)= \sum_{\substack{\overline{\rho}\in \NC_2(2q) \\ \rho=\gamma\overline{\rho} \\ \overline{\rho}\leq ker(rr)}}\sum_{\substack{A,B\in\rho \\ d_{T^\rho}(A,B)=2}}\sum_{\substack{i_{odd}=1,\dots,n \\ i_{even}=1,\dots,p \\ ker(i)=\rho_{A,B}}}\alpha_{a,b}.
\end{equation*}
We now appeal to the same arguments used in Theorem \ref{Lemma: Expansion of moments 4}. $\rho$ can be written as $\rho=\sigma\cup Kr(\sigma)$ with $\sigma\in \NC(\{2,\dots,2q\})$, further from Proposition \ref{Proposition: Equivalence of ker(r) and Ker(rr)} the condition $\overline{\rho}\leq ker(rr)$ becomes $\sigma/2 \leq ker(r)$. Thus
\begin{equation}\label{Aux 8}
\sum_{\substack{\pi\in \cP(2q)\\ T^\pi \text{is a 2-4 tree} }}\sum_{\substack{i_{odd}=1,\dots,n \\ i_{even}=1,\dots,p \\ ker(i)=\pi}}\E(x_i)= \sum_{\substack{\sigma\in \NC(\{2,\dots,2q\}) \\ \rho=\sigma\cup Kr(\sigma) \\ \sigma/2\leq ker(r)}}\sum_{\substack{A,B\in\rho \\ d_{T^\rho}(A,B)=2}}\sum_{\substack{i_{odd}=1,\dots,n \\ i_{even}=1,\dots,p \\ ker(i)=\rho_{A,B}}}\alpha_{a,b}.
\end{equation}
Note that if $A,B\in \rho$ are such that $d_{T^\rho}(A,B)=2$ then either $A,B\in \sigma$ or $A,B\in Kr(\sigma)$ because edges of $T$ always go from even to odd numbers or vice versa. Hence we can simplify the inner sum in the right hand side of Equation \ref{Aux 8} as follows
\begin{eqnarray*}
&&\sum_{\substack{A,B\in \rho \\ d_{T^\rho}(A,B)=2}}\sum_{\substack{i_{odd}=1,\dots,n \\ i_{even}=1,\dots,p \\ ker(i)=\rho_{A,B}}}\alpha_{a,b} \\
&=& \sum_{\substack{A,B\in \sigma \\ d_{T^\rho}(A,B)=2}}\sum_{\substack{i_{odd}=1,\dots,n \\ i_{even}=1,\dots,p \\ ker(i)=\rho_{A,B}}}\alpha_{a,b} + \sum_{\substack{A,B\in Kr(\sigma) \\ d_{T^\rho}(A,B)=2}}\sum_{\substack{i_{odd}=1,\dots,n \\ i_{even}=1,\dots,p \\ ker(i)=\rho_{A,B}}}\alpha_{a,b} \\
&=& \sum_{\substack{A,B\in \sigma \\ d_{T^\rho}(A,B)=2}}\alpha_{a,b}p^{\#(\sigma)-1}n^{\#(Kr(\sigma))} + \sum_{\substack{A,B\in Kr(\sigma) \\ d_{T^\rho}(A,B)=2}}\alpha_{a,b}p^{\#(\sigma)}n^{\#(Kr(\sigma))-1}+O(n^{q-1}), 
\end{eqnarray*}
where the last equality follow from
\begin{eqnarray*}
\sum_{\substack{i_{odd}=1,\dots,n \\ i_{even}=1,\dots,p \\ ker(i)=\rho_{A,B}}}1 &=& (p)_{\#(\sigma)-1}(n)_{\#(Kr(\sigma))}+O(n^{\#(\sigma)+\#(Kr(\sigma))-2}) \\
&=& p^{\#(\sigma)-1}n^{\#(Kr(\sigma))}+O(n^{q-1})
\end{eqnarray*}
whenever $A,B\in \sigma$. A similar statement holds when $A,B\in Kr(\sigma)$. Substituting this into Equation \ref{Aux 8} yields
\begin{multline}
\sum_{\substack{\pi\in \cP(2q)\\ T^\pi \text{is a 2-4 tree} }}\sum_{\substack{i_{odd}=1,\dots,n \\ i_{even}=1,\dots,p \\ ker(i)=\pi}}\E(x_i)= \\
\sum_{\substack{\sigma\in \NC(\{2,\dots,2q\}) \\ \sigma/2\leq ker(r) \\ \rho=\sigma\cup Kr(\sigma)}}\Bigg(\sum_{\substack{A,B\in \sigma \\ d_{T^\rho}(A,B)=2}}\alpha_{a,b}p^{\#(\sigma)-1}n^{\#(Kr(\sigma))} + \sum_{\substack{A,B\in Kr(\sigma) \\ d_{T^\rho}(A,B)=2}}\alpha_{a,b}p^{\#(\sigma)}n^{\#(Kr(\sigma))-1}+O(n^{q-1})\Bigg)
\end{multline}
Thus, we obtain
\begin{align*}
&\frac{1}{n^{q+1}}\sum_{\substack{\pi\in \cP(2q)\\ T^\pi \text{is a 2-4 tree} }}\sum_{\substack{i_{odd}=1,\dots,n \\ i_{even}=1,\dots,p \\ ker(i)=\pi}}\E(x_i) \\ 
&= \sum_{\substack{\sigma\in \NC(\{2,\dots,2q\}) \\ \sigma/2\leq ker(r) \\ \rho=\sigma\cup Kr(\sigma)}}\Bigg(\sum_{\substack{A,B\in \sigma \\ d_{T^\rho}(A,B)=2}}\alpha_{a,b}\frac{1}{n}\left(\frac{p}{n}\right)^{\#(\sigma)-1} + \sum_{\substack{A,B\in Kr(\sigma) \\ d_{T^\rho}(A,B)=2}}\alpha_{a,b}\frac{1}{n}\left(\frac{p}{n}\right)^{\#(\sigma)}\Bigg)+O(\frac{1}{n^2}),
\end{align*}
which simplifies to the desired result.
\end{proof}

Let us now look at the case when $T^\pi$ is a double unicycle of first type.

\begin{proposition}\label{Proposition: Infinitesimal contribution pi=q for double unicycle first type}
The following equation satisfies,
\begin{multline*}
\frac{1}{n^{q+1}}\sum_{\substack{\pi\in \cP(2q)\\ T^\pi \text{is a double unicycle} \\ \text{of first type} }}\sum_{\substack{i_{odd}=1,\dots,n \\ i_{even}=1,\dots,p \\ ker(i)=\pi}}\E(x_i)=
\sum_{\substack{\sigma\in \NC(2,\dots,2q) \\ \sigma/2\leq ker(r) \\ \rho=\sigma\cup Kr(\sigma)}}\frac{1}{n}\left(\frac{p}{n}\right)^{\#(\sigma)-1}\Bigg(\sum_{\substack{A,B\in \sigma \\ d_{T^\rho}(A,B)\neq 2}}1+\sum_{\substack{A,B\in Kr(\sigma) \\ d_{T^\rho}(A,B)\neq 2}}\frac{p}{n}\Bigg) \\
+ \frac{1}{n^{q+1}}\sum_{\substack{\sigma\in \NC(\{2,\dots,2q\}) \\ \pi=\sigma\cup Kr(\sigma) \\ \sigma/2\leq ker(r)}}\sum_{\rho\in L_\pi}\sum_{\substack{i_{odd}=1,\dots,n \\ i_{even}=1,\dots,p \\ ker(i)=\rho}}1 + O(\frac{1}{n^2}).
\end{multline*}
Here $L_\pi$ is defined as in Proposition \ref{Proposition: Order of sum when pi=q+1}.
\end{proposition}
\begin{proof}
First of all, by Proposition \ref{Proposition: contribution of double unicycle of first type} we have,
\begin{equation*}
\sum_{\substack{\pi\in \cP(2q)\\ T^\pi \text{is a double unicycle} \\ \text{of first type} }}\sum_{\substack{i_{odd}=1,\dots,n \\ i_{even}=1,\dots,p \\ ker(i)=\pi}}\E(x_i)= \sum_{\substack{\pi\in \cP(2q)\\ T^\pi \text{is a double unicycle} \\ \text{of first type}  \\ \overline{\pi}\leq ker(rr)}}\sum_{\substack{i_{odd}=1,\dots,n \\ i_{even}=1,\dots,p \\ ker(i)=\pi}}1
\end{equation*}
From Lemma \ref{Lemma: Characterization of double unicycle first type} we know $T^\pi$ is the result of joining two vertices that are at distance distinct from 2 of the graph $T^\rho$ which is a double tree. Further, $\rho=\gamma\overline{\rho}$ with $\overline{\rho}=\overline{\pi}\in \NC_2(2q)$. Hence proceeding as in Proposition \ref{Proposition: Infinitesimal contribution pi=q for 2-4 tree} we have
\begin{equation*}
\sum_{\substack{\pi\in \cP(2q)\\ T^\pi \text{is a double unicycle} \\ \text{of first type} }}\sum_{\substack{i_{odd}=1,\dots,n \\ i_{even}=1,\dots,p \\ ker(i)=\pi}}\E(x_i)= \sum_{\substack{\overline{\rho}\in \NC_2(2q)\\ \rho=\gamma\overline{\rho} \\ \overline{\rho}\leq ker(rr)}}\sum_{\substack{A,B\in \rho \\ d_{T^\rho}(A,B)\neq 2}}\sum_{\substack{i_{odd}=1,\dots,n \\ i_{even}=1,\dots,p \\ ker(i)=\rho_{A,B}}}1.
\end{equation*}
Now we use the same arguments as in Proposition \ref{Proposition: Infinitesimal contribution pi=q for 2-4 tree}, so that
\begin{equation}\label{Aux 9}
\sum_{\substack{\pi\in \cP(2q)\\ T^\pi \text{is a double unicycle} \\ \text{of first type} }}\sum_{\substack{i_{odd}=1,\dots,n \\ i_{even}=1,\dots,p \\ ker(i)=\pi}}\E(x_i)= \sum_{\substack{\sigma\in \NC(\{2,\dots,2q\})\\ \rho=\sigma\cup Kr(\sigma) \\ \sigma/2\leq ker(r)}}\sum_{\substack{A,B\in \rho \\ d_{T^\rho}(A,B)\neq 2}}\sum_{\substack{i_{odd}=1,\dots,n \\ i_{even}=1,\dots,p \\ ker(i)=\rho_{A,B}}}1.
\end{equation}
Observe that any two blocks $A,B\in\rho$ that are not at distance 2 in $T^\rho$ are either $A,B\in\sigma$, $A,B\in Kr(\sigma)$ or $A\in \sigma,B\in Kr(\sigma)$. So, we can decompose the inner sum of the right hand side in Equation \ref{Aux 9} as three sums, namely
$$\sum_{\substack{A,B\in \sigma \\ d_{T^\rho}(A,B)\neq 2}}\sum_{\substack{i_{odd}=1,\dots,n \\ i_{even}=1,\dots,p \\ ker(i)=\rho_{A,B}}}1, \phantom{a}\sum_{\substack{A,B\in Kr(\sigma) \\ d_{T^\rho}(A,B)\neq 2}}\sum_{\substack{i_{odd}=1,\dots,n \\ i_{even}=1,\dots,p \\ ker(i)=\rho_{A,B}}}1,\text{ and }\sum_{\substack{A\in\sigma , B\in Kr(\sigma) \\ d_{T^\rho}(A,B)\neq 2}}\sum_{\substack{i_{odd}=1,\dots,n \\ i_{even}=1,\dots,p \\ ker(i)=\rho_{A,B}}}1.$$
The first sum simplifies to 
$$\sum_{\substack{A,B\in \sigma \\ d_{T^\rho}(A,B)\neq 2}}p^{\#(\sigma)-1}n^{\#(Kr(\sigma))}+O(n^{q-1}),$$
by a similar argument to that used in Proposition \ref{Proposition: Infinitesimal contribution pi=q for 2-4 tree}. Similarly, the second sum simplifies to
$$\sum_{\substack{A,B\in Kr(\sigma) \\ d_{T^\rho}(A,B)\neq 2}}p^{\#(\sigma)}n^{\#(Kr(\sigma))-1}+O(n^{q-1}).$$
Finally in the last sum we can remove the condition $d_{T^\rho}(A,B)\neq 2$ as blocks in $\sigma$ and their Kreweras are always at odd distance in $T^\rho$, thus this sum becomes
$$\sum_{\substack{A\in\sigma , B\in Kr(\sigma)}}\sum_{\substack{i_{odd}=1,\dots,n \\ i_{even}=1,\dots,p \\ ker(i)=\rho_{A,B}}}1= \sum_{\pi\in L_\rho}\sum_{\substack{i_{odd}=1,\dots,n \\ i_{even}=1,\dots,p \\ ker(i)=\pi}}1.$$
Combining these facts with Equation \ref{Aux 9} yields
\begin{multline*}
\frac{1}{n^{q+1}}\sum_{\substack{\pi\in \cP(2q)\\ T^\pi \text{is a double unicycle} \\ \text{of first type} }}\sum_{\substack{i_{odd}=1,\dots,n \\ i_{even}=1,\dots,p \\ ker(i)=\pi}}\E(x_i)= \\
\frac{1}{n^{q+1}}\sum_{\substack{\sigma\in \NC(\{2,\dots,2q\})\\ \rho=\sigma\cup Kr(\sigma) \\ \sigma/2\leq ker(r)}}\sum_{\substack{A,B\in \sigma \\ d_{T^\rho}(A,B)\neq 2}}p^{\#(\sigma)-1}n^{\#(Kr(\sigma))}+\sum_{\substack{A,B\in Kr(\sigma) \\ d_{T^\rho}(A,B)\neq 2}}p^{\#(\sigma)}n^{\#(Kr(\sigma))-1} 
+\sum_{\pi\in L_\rho}\sum_{\substack{i_{odd}=1,\dots,n \\ i_{even}=1,\dots,p \\ ker(i)=\pi}}1+O(n^{q-1}) \\
= \sum_{\substack{\sigma\in \NC(\{2,\dots,2q\}) \\ \sigma/2\leq ker(r) \\ \rho=\sigma\cup Kr(\sigma)}}\sum_{\substack{A,B\in \sigma \\ d_{T^\rho}(A,B)\neq 2}}\frac{1}{n}\left(\frac{p}{n}\right)^{\#(\sigma)-1}+\sum_{\substack{A,B\in Kr(\sigma) \\ d_{T^\rho}(A,B)\neq 2}}\frac{1}{n}\left(\frac{p}{n}\right)^{\#(\sigma)} 
+\frac{1}{n^{q+1}}\sum_{\substack{\sigma\in \NC(\{2,\dots,2q\})\\ \rho=\sigma\cup Kr(\sigma) \\ \sigma/2\leq ker(r)}}\sum_{\pi\in L_\rho}\sum_{\substack{i_{odd}=1,\dots,n \\ i_{even}=1,\dots,p \\ ker(i)=\pi}}1+O(\frac{1}{n^2}),
\end{multline*}
which simplifies to the desired result.
\end{proof}

Finally, at this point it won't be necessary to look at the case when we have double unicycle of second type thanks to Proposition \ref{Proposition: contribution of double unicycle of second type}, however we will look at this case in subsection \ref{Section: Case of non zero psudovariance}. We are ready to prove our main results on the infinitesimal moments.

\begin{proof}[Proof of Theorem \ref{Lemma: Expansion of moments 5}]
From Lemma \ref{Lemma: Expansion of moments 2} and the fact that $\overline{C^\pi}$ has an edge of multiplicity 1 if and only if $\overline{T^\pi}$ does, we get,
$$\frac{1}{n}\E\circ \Tr(M^{(r_1)}\cdots M^{(r_q)})=\frac{1}{n^{q+1}}\sum_{\substack{\pi\in \cP(2q)\\ \overline{T^\pi} \text{ has no edges} \\\text{of multiplicity }1}}\sum_{\substack{i_{odd}=1,\dots,n \\ i_{even}=1,\dots,p \\ ker(i)=\pi}}\E(x_i).$$
By Lemma \ref{Lemma: First characterization of pi} we know $\#(\pi)\leq q+1$, thus
\begin{align*}
&\frac{1}{n}\E\circ \Tr(M^{(r_1)}\cdots M^{(r_q)})\\
&=\frac{1}{n^{q+1}}\sum_{\substack{\pi\in \cP(2q)\\ \#(\pi)=q+1}}\sum_{\substack{i_{odd}=1,\dots,n \\ i_{even}=1,\dots,p \\ ker(i)=\pi}}\E(x_i) 
+\frac{1}{n^{q+1}}\sum_{\substack{\pi\in \cP(2q)\\ \#(\pi)=q}}\sum_{\substack{i_{odd}=1,\dots,n \\ i_{even}=1,\dots,p \\ ker(i)=\pi}}\E(x_i)+\frac{1}{n^{q+1}}\sum_{\substack{\pi\in \cP(2q)\\ \#(\pi)\leq q-1}}\sum_{\substack{i_{odd}=1,\dots,n \\ i_{even}=1,\dots,p \\ ker(i)=\pi}}\E(x_i).
\end{align*}
But Proposition \ref{Proposition: N order of pi sum} says the last sum has order $O(\frac{1}{n^2})$, hence
\begin{multline*}
\frac{1}{n}\E\circ \Tr(M^{(r_1)}\cdots M^{(r_q)})=\frac{1}{n^{q+1}}\sum_{\substack{\pi\in \cP(2q)\\ \#(\pi)=q+1}}\sum_{\substack{i_{odd}=1,\dots,n \\ i_{even}=1,\dots,p \\ ker(i)=\pi}}\E(x_i)
+\frac{1}{n^{q+1}}\sum_{\substack{\pi\in \cP(2q)\\ \#(\pi)=q}}\sum_{\substack{i_{odd}=1,\dots,n \\ i_{even}=1,\dots,p \\ ker(i)=\pi}}\E(x_i)+O(\frac{1}{n^2}).
\end{multline*}
The second sum receives nonzero contributions only when $T^\pi$ is either a 2-4 tree, or a double unicycle of first type or a double unicycle of second type. Hence from Propositions \ref{Proposition: Infinitesimal contribution pi=q+1}, \ref{Proposition: Infinitesimal contribution pi=q for 2-4 tree} and \ref{Proposition: Infinitesimal contribution pi=q for double unicycle first type} we get

\begin{align*} 
&\frac{1}{n}\E\circ \Tr(M^{(r_1)}\cdots M^{(r_q)})= \frac{1}{n^{q+1}}\sum_{\substack{\sigma\in \NC(\{2,\dots,2q\}) \\ \pi=\sigma\cup Kr(\sigma) \\ \sigma/2\leq ker(r)}}\Bigg((p)_{\#(\sigma)}(n)_{\#(Kr(\sigma))}-\sum_{\rho\in L_\pi}\sum_{\substack{i_{odd}=1,\dots,n \\ i_{even}=1,\dots,p \\ ker(i)=\rho}}1\Bigg) \\ 
&+ \sum_{\substack{\sigma\in \NC(2,\dots,2q) \\ \sigma/2\leq ker(r) \\ \rho=\sigma\cup Kr(\sigma)}}\frac{1}{n}\left(\frac{p}{n}\right)^{\#(\sigma)-1}\Bigg(\sum_{\substack{A,B\in \sigma \\ d_{T^\rho}(A,B)=2}}\alpha_{a,b}+\sum_{\substack{A,B\in Kr(\sigma) \\ d_{T^\rho}(A,B)=2}}\frac{p}{n}\alpha_{a,b}+\sum_{\substack{A,B\in \sigma \\ d_{T^\rho}(A,B)\neq 2}}1+\sum_{\substack{A,B\in Kr(\sigma) \\ d_{T^\rho}(A,B)\neq 2}}\frac{p}{n}\Bigg) \\
&+ \frac{1}{n^{q+1}}\sum_{\substack{\sigma\in \NC(\{2,\dots,2q\}) \\ \pi=\sigma\cup Kr(\sigma) \\ \sigma/2\leq ker(r)}}\sum_{\rho\in L_\pi}\sum_{\substack{i_{odd}=1,\dots,n \\ i_{even}=1,\dots,p \\ ker(i)=\rho}}1 + \frac{1}{n^{q+1}}\sum_{\substack{\pi\in \cP(2q)\\ T^\pi \text{ is a double unicycle} \\ \text{of second type}}}\sum_{\substack{i_{odd}=1,\dots,n \\ i_{even}=1,\dots,p \\ ker(i)=\pi}}\E(x_i)+O(\frac{1}{n^2}). 
\end{align*}

By Proposition \ref{Proposition: contribution of double unicycle of second type} and basic algebra this simplifies to,
\begin{multline}\label{Aux 10}
\frac{1}{n}\E\circ \Tr(M^{(r_1)}\cdots M^{(r_q)})= \frac{1}{n^{q+1}}\sum_{\substack{\sigma\in \NC(\{2,\dots,2q\}) \\ \sigma/2\leq ker(r)}}(p)_{\#(\sigma)}(n)_{\#(Kr(\sigma))} \\
+ \sum_{\substack{\sigma\in \NC(2,\dots,2q) \\ \sigma/2\leq ker(r) \\ \rho=\sigma\cup Kr(\sigma)}}\frac{1}{n}\left(\frac{p}{n}\right)^{\#(\sigma)-1}\Bigg(\sum_{\substack{A,B\in \sigma \\ d_{T^\rho}(A,B)=2}}\alpha_{a,b}+\sum_{\substack{A,B\in Kr(\sigma) \\ d_{T^\rho}(A,B)=2}}\frac{p}{n}\alpha_{a,b}\Bigg) \\
+ \sum_{\substack{\sigma\in \NC(2,\dots,2q) \\ \sigma/2\leq ker(r) \\ \rho=\sigma\cup Kr(\sigma)}}\frac{1}{n}\left(\frac{p}{n}\right)^{\#(\sigma)-1}\Bigg(\sum_{\substack{A,B\in \sigma \\ d_{T^\rho}(A,B)\neq 2}}1+\sum_{\substack{A,B\in Kr(\sigma) \\ d_{T^\rho}(A,B)\neq 2}}\frac{p}{n}\Bigg)+O(\frac{1}{n^2}).
\end{multline}
Note that 
\begin{multline*}
(p)_{\#(\sigma)}(n)_{\#(Kr(\sigma))}=p^{\#(\sigma)}n^{\#(Kr(\sigma))} -\frac{\#(\sigma)(\#(\sigma)-1)}{2}p^{\#(\sigma)-1}n^{\#(Kr(\sigma))}\\
-\frac{\#(Kr(\sigma))(\#(Kr(\sigma))-1)}{2}p^{\#(\sigma)}n^{\#(Kr(\sigma))-1}+O(n^{\#(\sigma)+\#(Kr(\sigma))-2})
\end{multline*}
Substituting this into Equation \ref{Aux 10} and simplifying yields
\begin{multline}\label{Aux 11}
\frac{1}{n}\E\circ \Tr(M^{(r_1)}\cdots M^{(r_q)})= \\
\sum_{\substack{\sigma\in \NC(\{2,\dots,2q\}) \\ \sigma/2\leq ker(r)}}\left(\frac{p}{n}\right)^{\#(\sigma)}-\frac{\#(\sigma)(\#(\sigma)-1)}{2n}\left(\frac{p}{n}\right)^{\#(\sigma)-1}-\frac{\#(Kr(\sigma))(\#(Kr(\sigma))-1)}{2n}\left(\frac{p}{n}\right)^{\#(\sigma)} \\
+ \sum_{\substack{\sigma\in \NC(2,\dots,2q) \\ \sigma/2\leq ker(r) \\ \rho=\sigma\cup Kr(\sigma)}}\frac{1}{n}\left(\frac{p}{n}\right)^{\#(\sigma)-1}\Bigg(\sum_{\substack{A,B\in \sigma \\ d_{T^\rho}(A,B)=2}}(\alpha_{a,b}-1)+\sum_{\substack{A,B\in Kr(\sigma) \\ d_{T^\rho}(A,B)=2}}\frac{p}{n}(\alpha_{a,b}-1)\Bigg) \\
+ \sum_{\substack{\sigma\in \NC(2,\dots,2q) \\ \sigma/2\leq ker(r)}}\frac{1}{n}\left(\frac{p}{n}\right)^{\#(\sigma)-1}\left(\sum_{\substack{A,B\in \sigma}}1+\sum_{\substack{A,B\in Kr(\sigma)}}\frac{p}{n}\right)+O(\frac{1}{n^2}) \\
= \sum_{\substack{\sigma\in \NC(\{2,\dots,2q\}) \\ \sigma/2\leq ker(r)}}\left(\frac{p}{n}\right)^{\#(\sigma)}-\frac{\#(\sigma)(\#(\sigma)-1)}{2n}\left(\frac{p}{n}\right)^{\#(\sigma)-1}-\frac{\#(Kr(\sigma))(\#(Kr(\sigma))-1)}{2n}\left(\frac{p}{n}\right)^{\#(\sigma)} \\
+ \sum_{\substack{\sigma\in \NC(2,\dots,2q) \\ \sigma/2\leq ker(r) \\ \rho=\sigma\cup Kr(\sigma)}}\frac{1}{n}\left(\frac{p}{n}\right)^{\#(\sigma)-1}\Bigg(\sum_{\substack{A,B\in \sigma \\ d_{T^\rho}(A,B)=2}}(\alpha_{a,b}-1)+\sum_{\substack{A,B\in Kr(\sigma) \\ d_{T^\rho}(A,B)=2}}\frac{p}{n}(\alpha_{a,b}-1)\Bigg) \\
+ \sum_{\substack{\sigma\in \NC(2,\dots,2q) \\ \sigma/2\leq ker(r)}}\frac{1}{n}\left(\frac{p}{n}\right)^{\#(\sigma)-1}\left(\frac{\#(\sigma)(\#(\sigma)-1)}{2}+\frac{p}{n}\frac{\#(Kr(\sigma))(\#(Kr(\sigma))-1)}{2}\right)+O(\frac{1}{n^2}) \\
= \sum_{\substack{\sigma\in \NC(\{2,\dots,2q\}) \\ \sigma/2\leq ker(r)}}\left(\frac{p}{n}\right)^{\#(\sigma)}
+ \sum_{\substack{\sigma\in \NC(2,\dots,2q) \\ \sigma/2\leq ker(r) \\ \rho=\sigma\cup Kr(\sigma)}}\frac{1}{n}\left(\frac{p}{n}\right)^{\#(\sigma)-1}\Bigg(\sum_{\substack{A,B\in \sigma \\ d_{T^\rho}(A,B)=2}}(\alpha_{a,b}-1)+\sum_{\substack{A,B\in Kr(\sigma) \\ d_{T^\rho}(A,B)=2}}\frac{p}{n}(\alpha_{a,b}-1)\Bigg) +O(\frac{1}{n^2}).
\end{multline}
To conclude it is enough to relabel the second sum in terms of $\sigma/2$, so that $A,B\in\sigma$ becomes $A,B\in 2\sigma$, $A,B\in Kr(\sigma)$ becomes $A,B\in 2Kr(\sigma)-1$ and $\rho=\sigma\cup Kr(\sigma)$ becomes $2\sigma \cup (2Kr(\sigma)-1)$.

\end{proof}

\begin{proof}[Proof of Corollary \ref{Corollary: infinitesimal moments}]
From Theorem \ref{Lemma: Expansion of moments 5} and Corollary \ref{Corollary: Moments and cumulants} we get
\begin{multline*}
\mu^\prime(M^{(r_1)},\dots,M^{(r_q)})= 
\lim_{n\rightarrow\infty} \sum_{\substack{\sigma\in \NC(q) \\ \sigma\leq ker(r)}}n\left[\left(\frac{p}{n}\right)^{\#(\sigma)}-c^{\#(\sigma)}\right]+ \\
\sum_{\substack{\sigma\in \NC(q) \\ \sigma\leq ker(r)}}\left(\frac{p}{n}\right)^{\#(\sigma)-1}\Bigg(\sum_{\substack{A,B\in 2\sigma \\ d_{T^\rho}(A,B)=2}}(\alpha_{a,b}-1)+\sum_{\substack{A,B\in 2Kr(\sigma)-1 \\ d_{T^\rho}(A,B)=2}}\frac{p}{n}(\alpha_{a,b}-1)\Bigg) +O(\frac{1}{n}) \\
= \sum_{\substack{\sigma\in \NC(q) \\ \sigma\leq ker(r)}}\#(\sigma)c^\prime c^{\#(\sigma)-1}+ 
\sum_{\substack{\sigma\in \NC(q) \\ \sigma\leq ker(r)}}c^{\#(\sigma)-1}\Bigg(\sum_{\substack{A,B\in 2\sigma \\ d_{T^\rho}(A,B)=2}}(\alpha_{a,b}-1)+\sum_{\substack{A,B\in 2Kr(\sigma)-1 \\ d_{T^\rho}(A,B)=2}}c(\alpha_{a,b}-1)\Bigg).
\end{multline*}
This proves 1. To prove $2$, it is enough to note that when $\E|x_{1,2}|^4=2$ then $\alpha_{a,b}-1=0$ for any $a,b$ so that 
$$
\mu^\prime(M^{(r_1)},\dots,M^{(r_q)})=\sum_{\substack{\sigma\in \NC(q) \\ \sigma\leq ker(r)}}\#(\sigma)c^\prime c^{\#(\sigma)-1}.
$$
The last expression is the infinitesimal moment cumulant formula and the condition $\sigma\leq ker(r)$ is equivalent to vanishing of mixed infinitesimal cumulants, so we conclude 2.

\end{proof}


To prove Proposition~\ref{Proposition: inf moments of covariant matrix}, we first
establish the following lemma.

\begin{lemma}\label{lem:weighted-Kreweras-identity}
For every \(q\geq 1\) and \(c>0\), we have
\[
\sum_{\sigma\in\NC(q)}
c^{\#(\sigma)-1}
\sum_{A\in\operatorname{Kr}(\sigma)}
\binom{|A|}{2}
=
\sum_{\sigma\in\NC(q)}
c^{\#(\sigma)}
\sum_{A\in\sigma}
\binom{|A|}{2}.
\]
\end{lemma}

\begin{proof}
For \(\sigma\in\NC(q)\), set
\[
F(\sigma)
:=
\sum_{A\in\sigma}\binom{|A|}{2}.
\]
For \(0\leq k\leq q\), let
\[
\NC_k(q)
:=
\bigl\{
\sigma\in\NC(q):\#(\sigma)=k
\bigr\},
\qquad
\alpha_{q,k}
:=
\sum_{\sigma\in\NC_k(q)}F(\sigma),
\]
where we set \(\alpha_{q,0}=0\).

By \cite[Proposition~2.1]{Arizmendi12}, for
\(2\leq k\leq q-1\), the total number of blocks of size \(t\)
among all partitions in \(\NC_k(q)\) is
\[
\sum_{\sigma\in\NC_k(q)}
\#\bigl\{
A\in\sigma:|A|=t
\bigr\}
=
\binom{q}{k-1}
\binom{q-t-1}{k-2}.
\]
Consequently,
\begin{align*}
\alpha_{q,k}
&=
\sum_{t=2}^{q-k+1}
\binom{t}{2}
\sum_{\sigma\in\NC_k(q)}
\#\bigl\{
A\in\sigma:|A|=t
\bigr\} 
=
\binom{q}{k-1}
\sum_{t=2}^{q-k+1}
\binom{q-t-1}{k-2}
\binom{t}{2}.
\end{align*}

To evaluate the last sum, set
\[
j=t-2
\qquad\text{and}\qquad
m=q-k-1.
\]
Then
\begin{align*}
\sum_{t=2}^{q-k+1}
\binom{q-t-1}{k-2}
\binom{t}{2}
&=
\sum_{j=0}^{q-k-1}
\binom{j+2}{j}
\binom{q-j-3}{q-k-1-j} 
=
\sum_{j=0}^{m}
\binom{j+2}{j}
\binom{m-j+k-2}{m-j}.
\end{align*}
Using the shifted Vandermonde identity
\[
\sum_{j=0}^{m}
\binom{j+r}{j}
\binom{m-j+s}{m-j}
=
\binom{m+r+s+1}{m},
\]
with \(r=2\) and \(s=k-2\), we obtain
\begin{align*}
\sum_{t=2}^{q-k+1}
\binom{q-t-1}{k-2}
\binom{t}{2}
=
\binom{m+2+(k-2)+1}{m} 
=
\binom{m+k+1}{m} 
=
\binom{q}{q-k-1}
=
\binom{q}{k+1}.
\end{align*}
It follows that
\[
\alpha_{q,k}
=
\binom{q}{k-1}
\binom{q}{k+1},
\qquad
2\leq k\leq q-1.
\]

The same formula also holds for \(k=1\) and \(k=q\).
Indeed,
\[
\alpha_{q,1}
=
\binom{q}{2}
=
\binom{q}{0}\binom{q}{2},\quad
\text
{whereas}
\quad
\alpha_{q,q}=0
=
\binom{q}{q-1}\binom{q}{q+1},
\]
with the convention that
\(\binom{q}{r}=0\) whenever \(r\notin\{0,\ldots,q\}\).
Thus,
\[
\alpha_{q,k}
=
\binom{q}{k-1}
\binom{q}{k+1},
\qquad
0\leq k\leq q.
\]
In particular, we have 
$\alpha_{q,q-k}=\alpha_{q,k},\ 
0\leq k\leq q.
$
We now use the fact that the Kreweras complement is a bijection
satisfying
\[
\#\bigl(\operatorname{Kr}(\sigma)\bigr)
=
q+1-\#(\sigma).
\]
Hence it restricts, for every \(1\leq k\leq q\), to a bijection
\[
\operatorname{Kr}:
\NC_k(q)
\longrightarrow
\NC_{q+1-k}(q).
\]
Therefore,
\begin{align*}
\sum_{\sigma\in\NC(q)}
c^{\#(\sigma)-1}F\bigl(\operatorname{Kr}(\sigma)\bigr) &=
\sum_{k=1}^{q}
c^{k-1}
\sum_{\sigma\in\NC_k(q)}
F\bigl(\operatorname{Kr}(\sigma)\bigr) =
\sum_{k=1}^{q}
c^{k-1}\alpha_{q,q+1-k} \\
&=
\sum_{k=1}^{q}
c^{k-1}\alpha_{q,k-1} =
\sum_{j=0}^{q-1}
c^j\alpha_{q,j}.
\end{align*}
Since \(\alpha_{q,0}=\alpha_{q,q}=0\), the last expression equals
\[
\sum_{j=1}^{q}c^j\alpha_{q,j}
=
\sum_{\sigma\in\NC(q)}
c^{\#(\sigma)}F(\sigma).
\]
Recalling the definition of \(F\), we conclude that
\[
\sum_{\sigma\in\NC(q)}
c^{\#(\sigma)-1}
\sum_{A\in\operatorname{Kr}(\sigma)}
\binom{|A|}{2}
=
\sum_{\sigma\in\NC(q)}
c^{\#(\sigma)}
\sum_{A\in\sigma}
\binom{|A|}{2},
\]
as desired.
\end{proof}

\begin{proof}[Proof of Proposition \ref{Proposition: inf moments of covariant matrix}]
By Corollary~\ref{Corollary: infinitesimal moments}, we have
\begin{align}
\mu^\prime(M^q)
=
\sum_{\sigma\in\NC(q)}
\#(\sigma)c^\prime c^{\#(\sigma)-1}
+
(\frac{1}{2}\mathbb{E}|x_{1,2}|^4-1)
\sum_{\sigma\in\NC(q)}
c^{\#(\sigma)-1}
\Bigg(
\sum_{\substack{A,B\in 2\sigma\\
d_{T^\rho}(A,B)=2}}
1
+
c\sum_{\substack{A,B\in
2\operatorname{Kr}(\sigma)-1\\
d_{T^\rho}(A,B)=2}}
1
\Bigg).
\label{eq:initial-inf-moment-covariance}
\end{align}

Fix \(\sigma\in\NC(q)\). Any unordered pair of vertices belonging
to the same vertex class of the bipartite tree \(T^\rho\) and lying
at distance two has a unique common neighbor in the opposite vertex
class. Consequently,
\begin{align*}
&
\sum_{\substack{A,B\in 2\sigma\\
d_{T^\rho}(A,B)=2}}
1
+
c\sum_{\substack{A,B\in
2\operatorname{Kr}(\sigma)-1\\
d_{\bar{T}^\rho}(A,B)=2}}
1
=
\sum_{C\in 2\operatorname{Kr}(\sigma)-1}
\binom{\deg_{\overline{T^\rho}}(C)}{2}
+
c\sum_{C\in 2\sigma}
\binom{\deg_{\overline{T^\rho}}(C)}{2}.
\end{align*}
Here, \(\deg_{\overline{T^\rho}}(C)\) denotes the degree of the vertex \(C\)
in \(\overline{T^\rho}\), that is, the number of edges incident to \(C\).
By construction,
\[
\deg_{\overline{T^\rho}}(C)=|C|.
\]
It follows that
\begin{align}
&\sum_{\sigma\in\NC(q)}
c^{\#(\sigma)-1}
\Bigg(
\sum_{\substack{A,B\in 2\sigma\\
d_{T^\rho}(A,B)=2}}
1
+
c\sum_{\substack{A,B\in
2\operatorname{Kr}(\sigma)-1\\
d_{T^\rho}(A,B)=2}}
1
\Bigg)
\notag\\
&\qquad=
\sum_{\sigma\in\NC(q)}
c^{\#(\sigma)-1}
\Bigg(
\sum_{C\in 2\operatorname{Kr}(\sigma)-1}
\binom{|C|}{2}
+
c\sum_{C\in 2\sigma}
\binom{|C|}{2}
\Bigg)
\notag\\
&\qquad=
\sum_{\sigma\in\NC(q)}
c^{\#(\sigma)-1}
\sum_{C\in\operatorname{Kr}(\sigma)}
\binom{|C|}{2}
+
\sum_{\sigma\in\NC(q)}
c^{\#(\sigma)}
\sum_{C\in\sigma}
\binom{|C|}{2}.
\label{eq:inf-moments-covariance}
\end{align}
In the last equality, we used the cardinality-preserving bijections
\[
C\longmapsto 2C,
\qquad
C\longmapsto 2C-1,
\]
between the blocks of \(\sigma\) and \(2\sigma\), and between the
blocks of \(\operatorname{Kr}(\sigma)\) and
\(2\operatorname{Kr}(\sigma)-1\), respectively.

Applying Lemma~\ref{lem:weighted-Kreweras-identity} to the first
term on the right-hand side of
\eqref{eq:inf-moments-covariance}, we obtain
\begin{align*}
&\sum_{\sigma\in\NC(q)}
c^{\#(\sigma)-1}
\Bigg(
\sum_{\substack{A,B\in 2\sigma\\
d_{T^\rho}(A,B)=2}}
1
+
c\sum_{\substack{A,B\in
2\operatorname{Kr}(\sigma)-1\\
d_{T^\rho}(A,B)=2}}
1
\Bigg)
=
2\sum_{\sigma\in\NC(q)}
c^{\#(\sigma)}
\sum_{A\in\sigma}\binom{|A|}{2}.
\end{align*}
Substituting this identity into
\eqref{eq:initial-inf-moment-covariance} gives
\begin{align*}
\mu^\prime(M^q)
&=
\sum_{\sigma\in\NC(q)}
\#(\sigma)c^\prime c^{\#(\sigma)-1}
+
(\mathbb{E}|x_{1,2}|^4-2)
\sum_{\sigma\in\NC(q)}
c^{\#(\sigma)}
\sum_{A\in\sigma}\binom{|A|}{2}
\\
&=
\sum_{\sigma\in\NC(q)}
c^{\#(\sigma)-1}
\sum_{A\in\sigma}
\left(
c^\prime
+
c(\mathbb{E}|x_{1,2}|^4-2)\binom{|A|}{2}
\right).
\end{align*}
\end{proof}

\subsection{Nonzero pseudo variance}\label{Section: Case of non zero psudovariance}

In this section let us complete the case when $\mathbb{E}(x_{i,j}^2)\neq 0$. Note that this term appears only when $T^\pi$ is a double unicycle of second type. So, from now on $T^\pi$ is always a double unicycle of second type. 

Let $e_u$ and $e_v$ be such that $\{u,v\}\in \overline{\pi}$, i.e., $e_u$ and $e_v$ join the same pair of vertices of $T^\pi$. First of all, if we want $\mathbb{E}(x_i)\neq 0$ we require $\overline{\pi}\leq ker(rr)$ as otherwise by independence and centering of the variables we get $0$. However, even if $\overline{\pi}\leq ker(rr)$ we still might get $\mathbb{E}(x_i)=0$. Let us explore the cases in which this occurs.

If $\overline{e_u}$ is an edge not in the cycle of $\overline{T^\pi}$ then it must be $u$ and $v$ have distinct parity and $e_u$ and $e_v$ have the opposite orientation. Thus $e_u$ and $e_v$ have the same orientation in $C^\pi$ and this contributes a factor of the form $\mathbb{E}(x_{1,2}\overline{x_{1,2}})=1$. 

If $\overline{e_u}$ is within the cycle of $\overline{T^\pi}$ then it might be possible $u,v$ have the same parity or distinct parity. In the second case this leads to a factor of the form $\mathbb{E}(x_{1,2}\overline{x_{2,1}})=0$. So we only need to count double unicycles of second type where the first scenarios occurs. This is only possible when the cycle of $\overline{T^\pi}$ has even size (see for instance right hand side of Figure \ref{figure: double unicycle second type}).


\begin{figure}[htbp]
\centering
\sffamily\sansmath
\captionsetup{
  font={small,sf},
  labelfont={bf,sf},
  justification=raggedright,
  singlelinecheck=false
}

\begin{tikzpicture}[
  font=\small,
  vertex/.style={
    circle,draw,fill=white,
    inner sep=0pt,minimum size=4.5pt
  },
  directed edge/.style={
    -{Stealth[length=1.7mm,width=1.2mm]},
    line width=0.6pt,
    shorten <=1pt,
    shorten >=1pt
  }
]

\begin{scope}[xshift=-3cm]
  \foreach \k in {1,...,8} {
    \node[vertex] (v\k)
      at ({90-45*(\k-1)}:1.75) {};
    \node at ({90-45*(\k-1)}:2.02) {$\k$};
  }

  \foreach \k in {1,...,8} {
    \pgfmathtruncatemacro{\nextvertex}{mod(\k,8)+1}

    \draw[directed edge]
      (v\nextvertex) to[bend right=10] (v\k);

    \node at ({90-45*(\k-0.5)}:2.25) {$e_{\k}$};
  }

  \node at (0,-2.65) {$T=T_8$};
\end{scope}

\draw[directed edge]
  (-0.35,0) -- node[above] {$\pi$} (0.35,0);

\node[vertex,label=above:{$\{4,8\}$}]
  (D) at (1.35,1.25) {};

\node[vertex,label=above:{$\{1,5\}$}]
  (A) at (3.9,1.25) {};

\node[vertex,label=below:{$\{3,7\}$}]
  (C) at (1.35,-1.25) {};

\node[vertex,label=below:{$\{2,6\}$}]
  (B) at (3.9,-1.25) {};

\draw[directed edge] (A) to[bend right=18]
  node[midway,above] {$e_8$} (D);

\draw[directed edge] (A) to[bend left=18]
  node[midway,below] {$e_4$} (D);

\draw[directed edge] (D) to[bend right=18]
  node[midway,left] {$e_7$} (C);

\draw[directed edge] (D) to[bend left=18]
  node[midway,right] {$e_3$} (C);

\draw[directed edge] (C) to[bend left=18]
  node[midway,above] {$e_2$} (B);

\draw[directed edge] (C) to[bend right=18]
  node[midway,below] {$e_6$} (B);

\draw[directed edge] (B) to[bend left=18]
  node[midway,left] {$e_1$} (A);

\draw[directed edge] (B) to[bend right=18]
  node[midway,right] {$e_5$} (A);

\node at (2.625,-2.65) {$T^\pi$};

\end{tikzpicture}

\caption{
The oriented cycle $T=T_8$ and its quotient $T^\pi$ for
$\pi=\{\{1,5\},\{2,6\},\{3,7\},\{4,8\}\}$.
The edge labels and orientations are inherited from $T$, where
$e_u=(\gamma(u),u)$ and $\gamma=(1\,2\,\cdots\,8)$.
The quotient is a double unicycle of the second type.
}
\label{figure: double unicycle second type}
\end{figure}
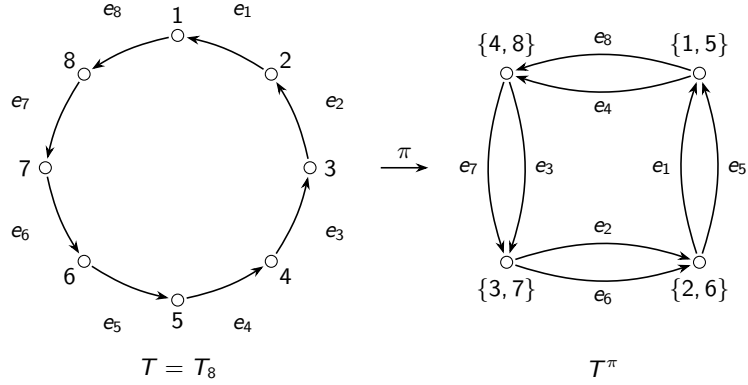

Suppose the graph $\overline{T^\pi}$ consists of a single cycle of size $q$, therefore, $\pi$ is uniquely determined as there is a unique way to traverse the entire graph $T^\pi$ starting from $e_{2q}$ and going down up to $e_1$. For instance if $q=4$ there is the unique partition $\pi=\{1,5\}\{2,6\}\{3,7\}\{4,8\}$ such that $\overline{T^\pi}$ has a single cycle of size $4$ (See figure \ref{figure: double unicycle second type}).

However the graph $T^\pi$ might also have subgraphs lying outside the cycle of $\overline{T^\pi}$. Observe that the subgraphs of $T^\pi$ outside the cycle of $\overline{T^\pi}$ are double trees, so these can be counted using non-crossing pairings. Therefore, to count all possible graphs $T^\pi$, we just need to choose the edges that will be within the cycle of $\overline{T^\pi}$ and consider non-crossing pairings for the rest of edges. An example that illustrates this construction can be seen in Example \ref{Example: double unicycle second type two}. 

\begin{example}\label{Example: double unicycle second type two}
Consider $q=7$ so that $2q=14$. We first choose the edges that will be within the cycle of $\overline{T^\pi}$. Let us choose $e_5,e_6,e_7,e_8,e_9,e_{10},e_{13},e_{14}$. We pair the rest of edges in a non-crossing way. This gives the non-crossing partition
$$\rho=\{1,4\}\{2,3\}\{5,6,7,8,9,10,13,14\}\{11,12\}.$$
The giant block $\{5,6,7,8,9,10,13,14\}$ corresponds to the edges in the cycle of $\overline{T^\pi}$ while the blocks $\{1,4\},\{2,3\},\{11,12\}$ are to the non-crossing pairings that correspond to the subgraph of $T^\pi$ outside the cycle. 
To get $\pi$ we first consider the Kreweras complement of $\rho$:
$$\gamma\rho^{-1}=(1,5)(2,4)(3)(6)(7)(8)(9)(10)(11,13)(12)(14).$$
The graph $T^{\gamma\rho^{-1}}$ is almost a double unicycle of second type, except that instead of a cycle with all edges of multiplicity $2$ we have a cycle with all edges of multiplicity $1$ consisting of the edges of the block $\{5,6,7,8,9,10,13,14\}$ (See Figure \ref{figure: double unicycle second type two}). At this point, by the preceding construction, there is a unique way to merge the vertices (blocks of $\gamma\rho^{-1}$) so that we get a cycle with all edges of multiplicity $2$ (See Figure \ref{figure: double unicycle second type two}), this merging produces $\pi$ and we denote this partition by ${\gamma\rho^{-1}}^\prime$.
Before moving on, let us point out that the number of edges lying within the cycle of $\overline{T^\pi}$ must be a multiple of $4$ as the cycle of $\overline{T^\pi}$ has even size and every edge has multiplicity $2$. 
\begin{figure}[p]
\centering
\sffamily\sansmath
\captionsetup{
  font={small,sf},
  labelfont={bf,sf},
  justification=raggedright,
  singlelinecheck=false
}

\tikzset{
  figtwo vertex/.style={
    circle,draw,fill=white,
    inner sep=0pt,minimum size=4pt
  },
  figtwo arrow/.style={
    -{Stealth[length=1.4mm,width=1mm]},
    line width=0.55pt,
    shorten <=1pt,
    shorten >=1pt
  }
}

\begin{minipage}{0.44\linewidth}
\centering
\begin{tikzpicture}[font=\small]

  \foreach \k in {1,...,14} {
    \coordinate (v\k)
      at ({90-360/14*(\k-1)}:1.62);
    \coordinate (p\k)
      at ({90-360/14*(\k-0.5)}:1.48);
  }

  \filldraw[
    draw=red!70!black,
    fill=red!6,
    line width=0.8pt
  ]
    (p5)--(p6)--(p7)--(p8)--(p9)--(p10)
    --(p13)--(p14)--cycle;

  \foreach \a/\b in {1/4,2/3,11/12}
    \draw[blue!70!black,line width=0.9pt]
      (p\a)--(p\b);

  \foreach \a/\b in {1/5,2/4,11/13}
    \draw[dashed,gray!80,line width=0.6pt]
      (v\a)--(v\b);

  \foreach \k in {1,...,14} {
    \pgfmathtruncatemacro{\nextvertex}{mod(\k,14)+1}

    \draw[figtwo arrow]
      (v\nextvertex) to[bend right=7] (v\k);

    \node[font=\scriptsize]
      at ({90-360/14*(\k-0.5)}:2.14)
      {$e_{\k}$};
  }

  \foreach \k in {1,...,14} {
    \node[figtwo vertex] at (v\k) {};
    \node at ({90-360/14*(\k-1)}:1.84) {$\k$};
  }

  \node at (0,-2.65)
    {(a) $T=T_{14}$, $\rho$ and $\sigma$};

\end{tikzpicture}
\end{minipage}\hfill
\begin{minipage}{0.54\linewidth}
\centering
\begin{tikzpicture}[font=\small]

  \node[figtwo vertex,label=right:{$\{1,5\}$}]
    (A) at (45:1.5) {};

  \node[figtwo vertex,label=right:{$\{6\}$}]
    (V6) at (0:1.5) {};

  \node[figtwo vertex,label=below right:{$\{7\}$}]
    (V7) at (-45:1.5) {};

  \node[figtwo vertex,label=below:{$\{8\}$}]
    (V8) at (-90:1.5) {};

  \node[figtwo vertex,label=below left:{$\{9\}$}]
    (V9) at (-135:1.5) {};

  \node[figtwo vertex,label=left:{$\{10\}$}]
    (V10) at (180:1.5) {};

  \node[figtwo vertex,label=left:{$\{11,13\}$}]
    (E) at (135:1.5) {};

  \node[figtwo vertex,label=above:{$\{14\}$}]
    (V14) at (90:1.5) {};

  \node[figtwo vertex,label=right:{$\{2,4\}$}]
    (B) at (1.85,1.95) {};

  \node[figtwo vertex,label=right:{$\{3\}$}]
    (C) at (2.5,2.8) {};

  \node[figtwo vertex,label=above left:{$\{12\}$}]
    (G) at (-1.8,1.95) {};

  \foreach \s/\t/\n/\pos in {
    A/V14/14/above,
    V14/E/13/above,
    E/V10/10/left,
    V10/V9/9/left,
    V9/V8/8/below,
    V8/V7/7/below,
    V7/V6/6/right,
    V6/A/5/right} {
    \draw[figtwo arrow]
      (\s) to[bend right=8]
      node[midway,\pos,font=\scriptsize]
        {$e_{\n}$}
      (\t);
  }

  \foreach \s/\t/\n/\pos in {
    B/A/1/left,
    A/B/4/right,
    C/B/2/left,
    B/C/3/right,
    G/E/11/left,
    E/G/12/right} {
    \draw[figtwo arrow]
      (\s) to[bend right=18]
      node[midway,\pos,font=\scriptsize]
        {$e_{\n}$}
      (\t);
  }

  \node at (0.2,-2.65)
    {(b) $T^\sigma$, $\sigma=\gamma\rho^{-1}$};

\end{tikzpicture}
\end{minipage}

\par\medskip

\begin{tikzpicture}[font=\small]

  \node[figtwo vertex,label=right:{$\{1,5,9\}$}]
    (A) at (1.1,0.8) {};

  \node[figtwo vertex,label=right:{$\{2,4\}$}]
    (B) at (1.9,1.8) {};

  \node[figtwo vertex,label=right:{$\{3\}$}]
    (C) at (2.5,2.65) {};

  \node[figtwo vertex,label=right:{$\{6,10\}$}]
    (D) at (1.1,-0.8) {};

  \node[figtwo vertex,label=left:{$\{7,11,13\}$}]
    (E) at (-0.9,-0.8) {};

  \node[figtwo vertex,label=above left:{$\{8,14\}$}]
    (F) at (-0.9,0.8) {};

  \node[figtwo vertex,label=below:{$\{12\}$}]
    (G) at (-1.65,-1.7) {};

  \foreach \s/\t/\n/\bend/\pos in {
    A/F/8/-18/above,
    A/F/14/18/below,
    F/E/7/-18/left,
    F/E/13/18/right,
    E/D/10/18/above,
    E/D/6/-18/below,
    D/A/9/18/left,
    D/A/5/-18/right} {
    \draw[figtwo arrow]
      (\s) to[bend left=\bend]
      node[midway,\pos,font=\scriptsize]
        {$e_{\n}$}
      (\t);
  }

  \foreach \s/\t/\n/\pos in {
    B/A/1/left,
    A/B/4/right,
    C/B/2/left,
    B/C/3/right,
    E/G/12/left,
    G/E/11/right} {
    \draw[figtwo arrow]
      (\s) to[bend right=18]
      node[midway,\pos,font=\scriptsize]
        {$e_{\n}$}
      (\t);
  }

  \node at (0.3,-2.4) {(c) $T^\pi$};

\end{tikzpicture}

{\small
\[
\begin{aligned}
\rho
  &=(1\,4)(2\,3)
    (5\,6\,7\,8\,9\,10\,13\,14)(11\,12),\\
\sigma
  &=\gamma\rho^{-1}
    =(1\,5)(2\,4)(3)(6)(7)(8)(9)(10)
      (11\,13)(12)(14),\\
\pi
  &=(1\,5\,9)(2\,4)(3)(6\,10)
    (7\,11\,13)(8\,14)(12).
\end{aligned}
\]
}

\caption{
The cycle $T=T_{14}$ and two successive quotient graphs,
where $\gamma=(1\,2\,\cdots\,14)$ and
$e_u=(\gamma(u),u)$.
In (a), the blue chords and the red polygon represent
the blocks of the edge partition $\rho$;
dashed chords represent the non-singleton blocks
of the vertex partition $\sigma$.
The vertex identifications $\sigma\leq\pi$
turn (b) into (c).
The graph $T^\pi$ is a double unicycle of the second type.
}
\label{figure: double unicycle second type two}
\end{figure}
\end{example}

The previous construction motivates the following notation.

\begin{notation}
\begin{enumerate}
    \item We denote by $\NC_2^\prime(2q)$ the set of non-crossing partitions, with any block of size $2$ except for a single block of size multiple of $4$.
    \item For $\rho\in \NC_2^\prime(2q)$ we let $A_\rho$ be the block of size $4$.
    \item If $A_\rho=\{i_1,\dots,i_{4s}\}$ then we let $\rho^\prime$ be the pairing with blocks the same as $\rho$ except the block $A_\rho$ which is replaced by the blocks $$\{i_1,i_{2s+1}\},\{i_2,i_{2s+2}\},\dots,\{i_{2s},i_{4s}\}.$$
\end{enumerate} 
\end{notation}

\begin{remark}\label{Remark: Double unicycle of second type}
Observe that $\rho^\prime$ is precisely the partition whose blocks $\{u,v\}$ are such that $e_u$ and $e_v$ join the same pair of vertices in the double unicycle graph of second type. So, if we want $\mathbb{E}(x_i)\neq 0$ we require $\rho^\prime \leq ker(rr)$ as discussed at the beginning of the section.
\end{remark}

\begin{proposition}\label{Proposition: Infinitesimal contribution pi=q for double unicycle second type}
The following equation satisfies,
\begin{multline*}
\frac{1}{n^{q+1}}\sum_{\substack{\pi\in \cP(2q)\\ T^\pi \text{is a double unicycle} \\ \text{of second type} }}\sum_{\substack{i_{odd}=1,\dots,n \\ i_{even}=1,\dots,p \\ ker(i)=\pi}}\E(x_i)= 
\sum_{\substack{\rho\in \NC_2^\prime(2q) \\ \rho^\prime \leq ker(rr)}}\frac{1}{n}\left( \frac{p}{n}\right)^{\#(Even(\gamma\rho^{-1}))-|A_\rho|/4}[\mathbb{E}(x_{1,2}^2)\mathbb{E}(\overline{x_{1,2}}^2)]^{|A_\rho|/4}+O(\frac{1}{n^2}).
\end{multline*}
Here $\#(Even(\gamma\rho^{-1}))$ denotes the number of cycles of $\gamma\rho^{-1}$ consisting only of even numbers, which is well defined because $\gamma\rho^{-1}$ has only cycles of purely even or odd numbers.
\end{proposition}
\begin{proof}
Thanks to the previous construction, we know the partitions $\pi$ such that $T^\pi$ is a double unicycle of second type (and $\mathbb{E}(x_i)\neq 0$) can be counted by the set $\NC_2^\prime(2q)$ where each $\rho\in \NC_2^\prime(2q)$ produces a double unicycle of second type $T^{{\gamma\rho^{-1}}^\prime}$. Further, thanks to Remark \ref{Remark: Double unicycle of second type}, if we require $\mathbb{E}(x_i)\neq 0$ we must consider $\rho^\prime\leq ker(rr)$. For each block $\{u,v\}\in {\gamma\rho^{-1}}^\prime$, there are three possible scenarios:
\begin{enumerate}
    \item Either $\overline{e_u}$ lies outside the cycle of $\overline{T^\pi}$ in which case the contribution to $\mathbb{E}(x_i)$ is $\mathbb{E}(x_{1,2}\overline{x_{1,2}})=1$,
    \item or, $\overline{e_u}$ lies within the cycle and $u$ and $v$ are both even in which case the contribution to $\mathbb{E}(x_i)$ is $\mathbb{E}(\overline{x_{1,2}}\overline{x_{1,2}})$,
    \item or, $\overline{e_u}$ lies within the cycle and $u$ and $v$ are both odd in which case the contribution to $\mathbb{E}(x_i)$ is $\mathbb{E}(x_{1,2}x_{1,2})$.
\end{enumerate}
The cycle of $\overline{T^{{\gamma\rho^{-1}}^\prime}}$ has length $|A_\rho|/2$, and half of the edges contribute $\mathbb{E}(x_{1,2}^2)$ and half of the edges contribute $\mathbb{E}(\overline{x_{1,2}}^2)$. Therefore,
\begin{align}\label{Aux 12}
\frac{1}{n^{q+1}}\sum_{\substack{\pi\in \cP(2q)\\ T^\pi \text{is a double unicycle} \\ \text{of second type} }}\sum_{\substack{i_{odd}=1,\dots,n \\ i_{even}=1,\dots,p \\ ker(i)=\pi}}\E(x_i) 
=\frac{1}{n^{q+1}}\sum_{\substack{\rho\in \NC_2^\prime(2q) \\ \rho^\prime \leq ker(rr)}}[\mathbb{E}(x_{1,2}^2)\mathbb{E}(\overline{x_{1,2}}^2)]^{|A_\rho|/4}\sum_{\substack{i_{odd}=1,\dots,n \\ i_{even}=1,\dots,p \\ ker(i)={\gamma\rho^{-1}}^\prime}}1.
\end{align}
The number of blocks of ${\gamma\rho^{-1}}^\prime$ is precisely $q$ because $T^{{\gamma\rho^{-1}}^\prime}$ is double unicycle. Further, the blocks of ${\gamma\rho^{-1}}^\prime$ consist of purely even or odd numbers, thus
$$\sum_{\substack{i_{odd}=1,\dots,n \\ i_{even}=1,\dots,p \\ ker(i)={\gamma\rho^{-1}}^\prime}}1=n^{\#(Odd({\gamma\rho^{-1}}^\prime))}p^{\#(Even({\gamma\rho^{-1}}^\prime))}+O(n^{q-1}),$$
where $\#Odd(\pi)$ and $\#(Even(\pi))$ denote the number of blocks of $\pi$ with only odd or even numbers respectively. Combining this with Equation \ref{Aux 12} yields 
\begin{align}\label{Aux 13}
&\frac{1}{n^{q+1}}\sum_{\substack{\pi\in \cP(2q)\\ T^\pi \text{is a double unicycle} \\ \text{of second type} }}\sum_{\substack{i_{odd}=1,\dots,n \\ i_{even}=1,\dots,p \\ ker(i)=\pi}}\E(x_i) \nonumber\\
&= 
\frac{1}{n^{q+1}}\sum_{\substack{\rho\in \NC_2^\prime(2q) \\ \rho^\prime \leq ker(rr)}}[\mathbb{E}(x_{1,2}^2)\mathbb{E}(\overline{x_{1,2}}^2)]^{|A_\rho|/4}n^{\#(Odd({\gamma\rho^{-1}}^\prime))}p^{\#(Even({\gamma\rho^{-1}}^\prime))}+O(\frac{1}{n^2}) \nonumber\\
&= \frac{1}{n}\sum_{\substack{\rho\in \NC_2^\prime(2q) \\ \rho^\prime \leq ker(rr)}}[\mathbb{E}(x_{1,2}^2)\mathbb{E}(\overline{x_{1,2}}^2)]^{|A_\rho|/4}\left(\frac{p}{n}\right)^{\#(Even({\gamma\rho^{-1}}^\prime))}+O(\frac{1}{n^2}).
\end{align}
To conclude it is easy to observe that
$$\#(Even({\gamma\rho^{-1}}^\prime))=\#(Even(\gamma\rho^{-1}))-|A_\rho|/4.$$
\end{proof}

As a consequence we get a full expansion up to order $\frac{1}{n}$ without the assumption $\mathbb{E}(x_{1,2}^2)=0$. We omit the proof as it follows the same as Theorem \ref{Lemma: Expansion of moments 5}, just considering the extra summation corresponding to the case where $T^\pi$ is double unicycle of second type whose contribution was found in Proposition \ref{Proposition: Infinitesimal contribution pi=q for double unicycle second type}.

\begin{lemma}\label{Lemma: Expansion of moments without pseudo variace zero}
Let $M^{(r_1)},\dots, M^{(r_q)}$ be a collection of independent covariance matrices, then
\begin{align*}
&\frac{1}{n}\E\circ \textup{\Tr}(M^{(r_1)}\cdots M^{(r_q)}) \\
&=\sum_{\substack{\sigma\in \NC(q) \\ \sigma\leq ker(r)}}\left(\frac{p}{n}\right)^{\#(\sigma)}+ 
\frac{1}{n}\sum_{\substack{\sigma\in \NC(q) \\ \sigma\leq ker(r)}}\left(\frac{p}{n}\right)^{\#(\sigma)-1}\Bigg(\sum_{\substack{A,B\in 2\sigma \\ d_{T^\rho}(A,B)=2}}(\alpha_{a,b}-1)+\sum_{\substack{A,B\in 2Kr(\sigma)-1 \\ d_{T^\rho}(A,B)=2}}\frac{p}{n}(\alpha_{a,b}-1)\Bigg) \\
&\hspace{1cm}+\frac{1}{n}\sum_{\substack{\rho\in \NC_2^\prime(2q) \\ \rho^\prime \leq ker(rr)}}\left( \frac{p}{n}\right)^{\#(Even(\gamma\rho^{-1}))-|A_\rho|/4}[\mathbb{E}(x_{1,2}^2)\mathbb{E}(\overline{x_{1,2}}^2)]^{|A_\rho|/4} +O(\frac{1}{n^2})
\end{align*}
where we recall the notation of Theorem \ref{Lemma: Expansion of moments 5} and Proposition \ref{Proposition: Infinitesimal contribution pi=q for double unicycle second type}.
\end{lemma}

\end{document}